\documentclass[1 leqno,11pt,psfig]{amsart}
\usepackage{amssymb, amsmath,amsmath,latexsym,amssymb,amsfonts,amsbsy, amsthm,mathtools,graphicx,CJKutf8,CJKnumb,CJKulem}
\usepackage{graphicx,epsfig}
\usepackage{graphicx}
\usepackage{amssymb}
\usepackage{graphicx}
\usepackage{graphicx,epsfig}
\usepackage{amsmath}
\usepackage{mathrsfs}
\usepackage{amssymb}
\usepackage[colorlinks=true,citecolor=blue,bookmarks=true]{hyperref}

\usepackage{amssymb,mathrsfs,amsfonts,amsmath,amsbsy,tikz}
\usepackage{fontenc}
\usepackage{textcomp}

\numberwithin{equation}{section}

\allowdisplaybreaks
\let\ep=\epsilon

\newcommand{\beq}{\begin{equation}}
\newcommand{\eeq}{\end{equation}}
\newcommand{\ben}{\begin{eqnarray}}
\newcommand{\een}{\end{eqnarray}}
\newcommand{\beno}{\begin{eqnarray*}}
\newcommand{\eeno}{\end{eqnarray*}}

\newtheorem{theorem}{Theorem}[section]
\newtheorem{definition}[theorem]{Definition}
\newtheorem{lemma}[theorem]{Lemma}

\newtheorem{remark}[theorem]{Remark}
\begin{document}
\begin{CJK*}{UTF8}{gkai}
\title[Nonlinear  stability of  non-rotating  gaseous stars]{Nonlinear  stability of  non-rotating  gaseous stars}

\author{Zhiwu Lin}
\address{School of Mathematics, Georgia Institute of Technology, 30332, Atlanta, GA, USA}
\email{zlin@math.gatech.edu}

\author{Yucong Wang}
\address{
School of Mathematics and Computational Science, Xiangtan University,  411105, Xiangtan, Hunan, P. R. China}
\email{yucongwang666@163.com}

\author{Hao Zhu}
\address{Department of Mathematics, Nanjing University,  210093, Nanjing, Jiangsu, P. R. China}
\email{haozhu@nju.edu.cn}

\date{\today}

\maketitle
\begin{abstract}
For the non-rotating gaseous stars modeled by the compressible Euler-Poisson system with general pressure law,
Lin and Zeng \cite{LZ2022} proved a turning point principle,
 which gives the sharp linear stability/instability  criteria for the non-rotating gaseous stars. In this paper,  we prove that the sharp linear stability  criterion for the non-rotating  stars also implies  nonlinear orbital  stability  against general perturbations provided the global weak solutions exist. If the perturbations are further restricted to be spherically symmetric,  then  nonlinear stability holds true unconditionally
 in the sense that the existence of global weak solutions near the  non-rotating star  can be proved.
\end{abstract}

\section{Introduction}%----------------------------introduction
The compressible Euler-Poisson (CEP) system
\begin{align}\label{EP}
\begin{cases}
\partial_t\rho + \text{div}(\rho v) = 0,\\
\partial_t(\rho v)+\text{div}(\rho v\otimes v)+\nabla p=-\rho\nabla V,\\
\Delta V=4\pi\rho,\;\;\lim_{|x|\to\infty}V(t,x)=0,
\end{cases}
\end{align}
describes the motion of a self-gravitating Newtonian gaseous star.
Here,  $t\geq0$, $x\in\mathbb{R}^3$, $\rho\geq0$ is the density, $p=P(\rho)$ is the pressure, $v=(v_1,v_2,v_3)\in\mathbb{R}^3$ is the velocity,
 and $V(x)=-\int_{\mathbb{R}^3}\frac{\rho(y)}{|x-y|}dy$ is the self-consistent
gravitational potential.
The initial data
satisfies
\begin{align}\label{inaprx}
(\rho(0,x),(\rho v)(0,x))\rightarrow (0,0) \quad \text{as} \quad |x|\rightarrow\infty.
\end{align}
\if0
Denote
\begin{align}
\hat{M}=\int_{\mathbb{R}^3}\hat{\rho}(x)dx,\quad\hat{V} =4\pi \Delta^{-1} \hat{\rho}.
\end{align}
\fi
In this paper, we assume that the pressure function $p=P(\rho)$ satisfies
\begin{align}\label{P1}
P \in C^1(0,\infty),\quad P'>0,
\end{align}
and there exist $\gamma_{0},\gamma_1\in(\frac{6}{5},2)$ such that
\begin{align}\label{P2}
\lim_{s\rightarrow 0^+}s^{1-\gamma_{0}}P'(s)=K_{0}>0,
\end{align}
\begin{align}\label{P3}
\lim_{s\rightarrow \infty}s^{1-\gamma_1}P'(s)=K_1>0.
\end{align}
The assumptions \eqref{P2} and \eqref{P3} imply that the pressure $P(\rho)\approx K_{0}\rho^{\gamma_{0}}$ for $\rho$ near 0 and $P(\rho)\approx K_1\rho^{\gamma_1}$ for $\rho$ large enough. By a non-rotating star, we mean a static equilibrium $(\rho,v=0)$ of \eqref{EP}. It is known \cite{GNN1981} that if $\rho$ is   compactly supported for a non-rotating star $(\rho,0)$, then $\rho$ must be radially
symmetric.
The total energy is
\begin{align*}
E(\rho,v)=\int_{\mathbb{R}^3}\left(\frac{1}{2}\rho|v|^2+\Phi(\rho)-\frac{1}{8\pi}|\nabla V|^2\right)dx,\quad(\rho,v)\in X,
\end{align*}
which is the sum of  the kinetic energy and the potential energy (enthalpy and gravity). Here,
 the enthalpy function $\Phi(\rho) > 0$ is defined by
\begin{align*}
\Phi(0)=\Phi'(0)=0,\quad \Phi'(\rho)=\int_{0}^{\rho}\frac{P'(s)}{s}ds,
\end{align*}
and the energy space is given by
\begin{align}\label{def-X-mu}
X=\left\{(\rho,v)|\int_{\mathbb{R}^3}\rho dx<\infty, \int_{\mathbb{R}^3}\Phi(\rho )dx<\infty,\int_{\mathbb{R}^3}\rho v^2dx<\infty\right\}.
\end{align}

Much progress has been made on the stability of non-rotating gaseous stars in both linear and nonlinear senses.
%Below we briefly review some recent literature and state
%our main results.
For  the polytropic stars, it is shown in \cite{LSS1997} that they are linearly stable for $\gamma\in(4/3,2)$ and linearly unstable for $\gamma\in(1,4/3)$.
For general equations of state,   there exists $\mu_{\textrm{max}}\in(0,\infty]$ such that for any center density $\rho_{\mu} (0) = \mu\in
(0, \mu_{\textrm{max}})$, there exists a unique non-rotating star $(\rho_{\mu},0)$ such that the density $\rho_{\mu}(|x|)$ is supported inside a ball
with radius $  R (\mu)=R_{\mu} < \infty$ (see, e.g. \cite{LZ2022}). Moreover, $\mu_{\textrm{max}} = \infty$ when $\gamma_{0} \geq 4/3$
(\cite{HU2003,MAKINO1984,SW2019R,RTRG2013}).
Denote
$$M (\mu) =M_\mu= \int_{\mathbb{R}^3}\rho_\mu dx$$
to be the total mass of the star. Then
the sharp linear stability/instability  criteria of the family of non-rotating stars $(\rho_{\mu},0)$ with $\mu\in(0,\mu_{\textrm{max}})$ are entirely described  by the following turning point principle.

\vspace{3mm}

\noindent{\bf Theorem A \cite{LZ2022}.}
{\it
The linear stability of $(\rho_\mu,0)$ is fully determined by the mass-radius curve parameterized
by $\mu\in(0,\mu_{\rm{max}})$.
Let $n^{u}\left(  \mu\right)
$ be the number of unstable modes.
For $\mu\ll1$, we have
\begin{equation*}
n^{u}\left(  \mu\right)  =%
\begin{cases}
1 & \text{when }\gamma_{0}\in\left(  \frac{6}{5},\frac{4}{3}\right), \\
0 & \text{ when }\gamma_{0}\in\left(  \frac{4}{3},2\right).
\end{cases}
 \label{formula-unstable-mode-small-mu}%
\end{equation*}
The number $n^{u}\left(  \mu\right)  \ $can only change at mass extrema (i.e.
maxima or minima of $M\left(  \mu\right)  $). For increasing $\mu$, at a mass
extrema point where $M'(\mu)$ changes sign, $n^{u}\left(  \mu\right)  $ increases by $1$ if $M^{\prime}%
(\mu)R^{\prime}\left(  \mu\right)  $ changes from $-$ to $+$ (i.e. the
mass-radius curve bends counterclockwise) and $n^{u}\left(  \mu\right)  $
decreases by $1$ if $M^{\prime}(\mu)R^{\prime}\left(  \mu\right)  $ changes
from $+$ to $-\ $(i.e. the mass-radius curve bends clockwise).
}
\vspace{3mm}

Global existence of solutions of  the initial value problem \eqref{EP}-\eqref{inaprx}  has attracted much attention but is still widely open for non-spherically-symmetric initial data. Local well-posedness of strong solutions was proved in \cite{LG2016,LXZ2014wellposedness,MT1986}.  More recently,
 the existence of global weak solutions of the CEP equations  with large
initial data of spherical symmetry was proved  for the  polytropic stars in \cite{CHWY2021}  and  for general equations of state including the white dwarf stars in \cite{CHWW2023}.

In the literature (\cite{LTSJ2008,RG2003}), the nonlinear stability of the non-rotating stars is conditional in the sense that
the existence of a global weak solution  to the initial value problem \eqref{EP}-\eqref{inaprx} is assumed.
Rein \cite{RG2003} constructed the non-rotating stars as global minimizers of the energy functional under the mass-conserved constraint with $\gamma_0,\gamma_1\in(4/3,\infty)$, and then
proved  nonlinear orbital  stability for such  non-rotating stars by using the compactness of minimizing sequences under
the assumption that the minimizer is unique up to spatial translations.
In \cite{LTSJ2008}, Luo and Smoller proved nonlinear orbital stability  of the white dwarf stars by a similar approach   when the
total mass is beneath a critical mass. Global   minimizer of the energy functional in the mass-conserved set is
unique up to  spatial translations for a class of non-rotating stars including the polytropic stars \cite{RG2003} and white dwarf stars \cite{LY1987}. However,
Rein constructed a pressure law and provided numerical evidence to show that   the minimizers are
not unique \cite{RG2003}. In the case that the minimizers are not unique but isolated  up to
spatial translations, nonlinear orbital stability still holds true provided the perturbed solution of the
CEP system is sufficiently continuous in time so that it cannot jump from one
minimizer to the next \cite{RG2003}.
%i.e. $d((\rho(t),v(t)),(\rho_\mu,0))\leq C [d(({\rho}(0),{v}(0)),(\rho_\mu,0))+|{M}-M({\mu})|^q]$  instead of
%the usual conclusion of Lyapunov stability.
Nonlinear instability of the polytropic stars for $\gamma\in[6/5,4/3]$
has been proved in \cite{DLYY2002,JJ2008,JJ2014}.
Moreover, the strong instability of the polytropic stars is proved in \cite{CCL2023} for
 $\gamma\in(6/5,4/3]$ in the sense that there exists nearby radially symmetric initial
 data with global weak solutions whose support expanding to infinity as $t\rightarrow\pm\infty$.
%Let us also mention the recent result on the nonlinear
%stability of the Goldreich-Weber solution under small spherically symmetric perturbations for the
%adiabatic exponent $\gamma= 4/3$ \cite{HadzicJang2018Nonlinear}.

Our first main result is that  except for the mass extrema points,  the linear
stability criterion for the non-rotating stars in Theorem A is also conditionally true at the nonlinear level against general (not necessarily spherically symmetric) perturbations. Here, ``conditionally" means the following  assumptions (1)-(2).

\begin{theorem}\label{main}
Suppose that $P$ satisfies \eqref{P1}-\eqref{P3}.
For any non-rotating stars $(\rho_\mu,0)$ satisfying  $n^u(\mu)=0$ and $M'(\mu)\neq0$,  assume that

$(1)$ for the initial data $(\rho(0),v(0))\in X$, a finite-energy weak solution $(\rho,v)$ to the CEP system \eqref{EP}-\eqref{inaprx} exists globally in the sense of Definition \ref{weaksol},

%$(2)$ the total energy $E(t)$ is non-increasing with respect to $t$,\\
$(2)$ the distance functional $d((\rho(t),v(t)),(\rho_\mu,0))$  is continuous respect to $t$.
\\
Let $1<q<2$. For any $\epsilon>0$, there exists  $\delta>0$ such that if the initial data satisfies
\begin{align}\label{initial data-1}
\inf_{y\in \mathbb{R}^3}d(({\rho}(0),{v}(0)),(\rho_\mu(x+y),0))+|{M}-M({\mu})|^q<\delta,
\end{align}
then
\begin{align*}
\inf_{y\in \mathbb{R}^3}d((\rho(t),v(t)),(\rho_\mu(x+y),0))<\epsilon
\end{align*}
 for all $t\geq 0$, where  $M  = \int_{\mathbb{R}^3}\rho(0) dx$, $X$ and  $d$ are defined in \eqref{def-X-mu} and \eqref{distancefun}, respectively.
\end{theorem}

Next,
we consider a class of
the pressure functions satisfying\smallskip \\
(C1) The pressure $p=P(\rho)\in C^1([0,\infty))\cap C^4((0,\infty))$ and
\begin{align*}
{P(\rho)'>0,\quad  \rho P''(\rho)+2P'(\rho)>0\quad \text{for}\quad \rho>0.}
\end{align*}
(C2) There exists a constant $\rho_*>0$ such that
\begin{align*}
P(\rho)=\kappa_0\rho^{\gamma_0}(1+P_0(\rho))\quad\text{for}\quad \rho\in [0,\rho_*),
\end{align*}
where $\kappa_0>0$, $\gamma_0\in(6/5,2)$ and $P_0(\rho)\in C^4((0,\infty))$. There exists $C_*>0$ such that $|P_0^{(j)}(\rho)|\leq C_*\rho^{2\theta_0-j}$ for $\rho\in(0,\rho_*)$ and $j=0,\cdots, 4$, where $\theta_0=(\gamma_0-1)/2$.\\
(C3) There exists a constant $\rho^*>\rho_*>0$ such that
\begin{align*}
P(\rho)=\kappa_1\rho^{\gamma_1}(1+P_1(\rho))\quad \text{for}\quad \rho\in [\rho^*,\infty),
\end{align*}
where $\kappa_1>0$, {$\gamma_1\in({6\over5},\gamma_0]$} and $P_1(\rho)\in C^4((0,\infty))$. There exists $C^*>0$ such that $|P_1^{(j)}(\rho)|\leq C^*\rho^{-\ep-j}$ for $\rho\in[\rho^*,\infty)$ and $j=0,\cdots, 4$, where $\ep>0$. \smallskip

Besides the polytropic stars, a typical example is the white dwarf stars, since  the pressure $P_w(\rho)$
satisfies the conditions  (C1)-(C3) with $\gamma_0={5\over3}, \theta_0={1\over3}, \gamma_1={4\over3},\ep={2\over3}$. Here, $P_w(\rho)=Af(x)$ and $\rho=B x^3$, where $A, B$ are two constants and
$$f(x)=x\sqrt{1+x^2}(2x^2-3)+3\ln(x+\sqrt{1+x^2})=8\int_0^x {y^4\over \sqrt{1+y^2}}dy.$$

The second result of this paper is that if  the perturbations are restricted to be spherically symmetric, then
 we can get much stronger results than Theorem \ref{main} for a class of pressure laws satisfying (C1)-(C3). That is,
 the non-rotating stars are nonlinearly stable unconditionally.
Here,
``unconditionally" means that if the spherically symmetric initial data  is close to the non-rotating star under the
distance $d$, then
 there exists a global weak solution to the initial value problem \eqref{EP}-\eqref{inaprx} which remains close to the non-rotating star
 under the same distance.

  %By a similar argument to Theorem \ref{unconditionalTHM} we can prove that

  %is close to the non-rotating star  under the distance $d$, then\\
 %(i) there exists a global finite-energy solution to the Euler-Poisson system \eqref{EP}-\eqref{inaprx} with spherical symmetry,\\
  %(ii) the non-rotating star is nonlinearly stable.

%Then we have the following result.

\begin{theorem}\label{unconditionalTHM}
Suppose that $p=P(\rho)$ satisfies (C1)-(C3).
For any non-rotating star $(\rho_\mu,0)$ satisfying  $n^u(\mu)=0$ and $M'(\mu)\neq0$, there exists  $\delta>0$ such that if
 the spherically symmetric  initial data $(\rho(0),v(0))\in X$ satisfies
 \begin{align*}
d(({\rho}(0),{v}(0)),(\rho_\mu,0))+|{M}-M({\mu})|^q=\tilde \delta<\delta,
\end{align*}
then

$({\rm i})$  there exists a global finite-energy weak solution $(\rho,v)$ to the CEP system \eqref{EP}-\eqref{inaprx} with spherical symmetry in the sense of Definition \ref{weaksol},

$({\rm ii})$
\begin{align}\label{distance-t-spherical symmetric perturbations}
\sup_{ t> 0} d((\rho(t),v(t)),(\rho_\mu,0))=O(\tilde \delta),
\end{align}
 where $1<q<2$, $M  = \int_{\mathbb{R}^3}\rho(0) dx$, $X$ and  $d$ are defined in \eqref{def-X-mu} and \eqref{distancefun}, respectively.
\end{theorem}

\begin{remark}\label{spherically symmetric case remark}
% (1) Unlike general perturbations in Theorem \ref{main}, we do not need to assume that the distance functional $d((\rho(t),v(t)),(\rho_\mu,0))$  is continuous with respect to $t$ under the spherically symmetric perturbations in Theorem \ref{unconditionalTHM}.

% (2) For the polytropic equation of state, the global existence of spherically symmetric finite-energy solution to the Euler-Poisson system \eqref{EP}-\eqref{inaprx} with the initial data $(\rho(0),v(0))\in X$ is proved in \cite{CHWY2021}. The existence is proved under the condition that $M<M_c(\gamma)$ for some critical mass $M_c(\gamma)$ when $\gamma\in({6\over5},{4\over3}]$ and unconditionally when $\gamma>{4\over3}$.

 %(3) For a class of general pressure laws,  if the non-rotating star $(\rho_\mu,0)$ is spectrally stable and $M'(\mu)\neq0$, then a similar argument to Theorem \ref{unconditionalTHM} ensures that
 %a global finite-energy solution  with spherical symmetry exists near the  non-rotating star
 %and the non-rotating star is nonlinearly stable, see Remark \ref{general pressure law} for more details. Here, the existence is obtained without the assumption that the total mass is   beneath some undetermined critical mass.

Note that  we do not need the condition that the total mass is   beneath some  critical mass when ${6\over5}<\gamma_1\leq{4\over3}$, which  is assumed in \cite{CHWY2021,CHWW2023} to prove the existence of a global  weak solution.
In fact, this mass restriction is mainly used to obtain a uniform bound of internal energy $\int\Phi(\rho)dx$ for the approximate solutions  by some  Sobolev inequalities and a continuous argument in \cite{CHWY2021,CHWW2023}.
Our analysis is carried out around the linearly stable non-rotating star, and we use the nonlinear stability
 estimates \eqref{symmetric-p-estimate} to obtain the uniform bound $\int\Phi(\rho)dx$ for the approximate solutions. Then the existence of a global  weak solution follows from the compactness arguments and compensated compactness frameworks in \cite{CHWY2021,CHWW2023}.

\end{remark}
 Some ideas in the proof are collected as follows.
We define the energy-Casimir  functional
%in \eqref{Energy-Casimir functional}
\begin{align*}
H(\rho,v)
=\int_{\mathbb{R}^3}\left(\frac{1}{2}\rho|v|^2+\Phi(\rho)-\frac{1}{8\pi}|\nabla V|^2\right)dx-V_{\mu}(R_{\mu})\int_{\mathbb{R}^3}\rho dx,\quad(\rho,v)\in X,
\end{align*}
and the distance functional %in \eqref{distancefun}.
\begin{align*}
d((\rho,v),(\rho_\mu,0))&=\frac{1}{2}\int_{\mathbb{R}^3}\rho|v|^2dx
+\int_{B_\mu}\left(\Phi(\rho)-\Phi(\rho_\mu)-\Phi'(\rho_\mu)(\rho-\rho_\mu)\right) dx\\
\notag&\quad+\int_{B_\mu^c}\Phi(\rho) dx+\frac{1}{8\pi}\int_{\mathbb{R}^3}|\nabla V_{\text{in}}-\nabla V_\mu|^2dx+\int_{B_\mu^c} (V_\mu-V_\mu(R_{\mu}))\rho dx\\
&:=d_1+d_2+d_3
+d_4+d_5,
\end{align*}
where $B_{\mu}$ is the support of $\rho_\mu$ and $B_{\mu}^c$ is the complement set of $B_{\mu}$.
Here, we decompose the density $\rho$ into two parts $\rho_{\text{in}}$ and $\rho_{\text{out}}$ (i.e. inside and outside the support of the background non-rotating star), $\Delta V_{\text{in}}=4\pi\rho_{\text{in}}$ and $\Delta V_{\text{out}}=4\pi\rho_{\text{out}}$.
Then the relative energy-Casimir functional can be expressed by the distance functional in the form
\begin{align}\label{H-rho-vHrhomu0}
H(\rho,v)-H(\rho_{\mu},0)=&d_1+d_2+d_3
-d_4+d_5\\\nonumber
&-\frac{1}{8\pi}\int_{\mathbb{R}^3}|\nabla V_{\text{out}}|^2dx
-\frac{1}{4\pi}\int_{\mathbb{R}^3}\nabla V_{\text{out}}\cdot(\nabla V_{\text{in}}-\nabla V_{\mu})dx.
\end{align}
The relative energy-Casimir functional can be controlled by
\begin{align*}
H(\rho,v)-H(\rho_{\mu},0)\geq& d_1+d_3+d_5
+(d_2-d_4)-o(d),
\end{align*}
where the last two terms in the right hand side of \eqref{H-rho-vHrhomu0} are shown to be $o(d)$ in Lemma \ref{high order term EC functional}.
To obtain a lower bound of the relative energy-Casimir functional in terms of the distance, we need to deal with  the key functional above
\begin{align*}
F(\tilde \rho_{\text{in}})=d_2-d_4
=&\int_{B_\mu}\left(\Phi(\rho_\mu+\tilde\rho_{\text{in}})-\Phi(\rho_\mu)-\Phi'(\rho_\mu)\tilde\rho_{\text{in}}\right) dx-\frac{1}{8\pi}\int_{\mathbb{R}^3}|\nabla V_{\text{in}}-\nabla V_{\mu}|^2dx
\end{align*}
with $\tilde \rho_{\text{in}}=\rho_{\text{in}}-\rho_\mu\in L^2_{\Phi''(\rho_\mu)}(B_\mu)$.
We first try to study the Taylor
expansion of $F(\tilde\rho_{\text{in}})$ directly and apply the sharp linear stability condition that $F''(0)=L_\mu\geq0$ under the mass constraint $\int_{B_{\mu}} \tilde \rho_{\text{in}} dx=0$, but its regularity is not enough.
More precisely, $F\notin  C^2(L^2_{\Phi''(\rho_\mu)}(B_\mu))$ as required for the second order Taylor
expansion of $F$.
Here, $L_\mu$
is defined by
\begin{align}\label{def-L-mu}
L_\mu=\Phi''(\rho_\mu)-{4\pi}(-\Delta)^{-1} :\;\; L^2_{\Phi''(\rho_\mu)}(B_\mu)\to (L^2_{\Phi''(\rho_\mu)}(B_\mu))^*.
\end{align}
To get sufficient regularity of the functional so that  the Taylor
expansion   can be carried out, we introduce a dual functional
\begin{align*}
B(\tilde{V}_{\text{in}})={1\over 8\pi}\int_{\mathbb{R}^3}|\nabla\tilde{V}_{\text{in}}|^2dx+\int_{B_{\mu}}\Psi_{\rho_{\mu}}^*(P\tilde{V}_{\text{in}}-\tilde{V}_{\text{in}}) dx
\end{align*}
with $\tilde{V}_{\text{in}}=V_{\text{in}}-V_{\mu}\in \dot{H}^1(\mathbb{R}^3)$,
where
$\Psi_{\rho_{\mu}}^*$ is the Legendre transformation
of \begin{align*}
\Psi_{\rho_{\mu}(r)}(\tau)=\Phi(\tau+\rho_{\mu}(r))-\Phi(\rho_{\mu}(r))-\Phi'(\rho_{\mu}(r))\tau,
\end{align*}
and
$P\tilde{V}_{\text{in}}=\frac{\int_{B_{\mu}}\frac{\tilde{V}_{\text{in}}}{\Phi''(\rho_{\mu})}dx}{\int_{B_{\mu}}\frac{1}{\Phi''(\rho_{\mu})}dx}$.
The projection $P$ is added during the duality process by using the mass constraint.
We can prove that $B\in C^2(\dot{H}^1(\mathbb{R}^3))$ via a careful analysis by means of the index scope $\gamma_0, \gamma_1\in ({6\over5},2)$. Moreover, we have
\begin{align*} B''(0)={1\over 4\pi}\tilde L_\mu,
\end{align*}
where
$$\tilde L_\mu=-\Delta-{4\pi\over \Phi''(\rho_{\mu})}(I-P): \dot{H}^1(\mathbb{R}^3)\to\dot{H}^{-1}(\mathbb{R}^3).$$
A key point is that the projection term $P$ added here helps us
to relate
$\tilde{L}_\mu$ to the sharp linear stability condition $L_{\mu}|_{Z_\mu}\geq0$, where $Z_\mu$ is the mass-preserving space.  %The sharp linear stability condition is that $L_\mu\geq0$ under the mass constraint $\int_{B_{\mu}} \rho dx=0$.
Indeed,
the linear stability condition $L_{\mu}|_{Z_\mu}\geq0$ is equivalent to $\tilde L_\mu\geq0$. To get the positivity of the second order variation of the dual functional  $B$ at $0$ (i.e. $\tilde{L}_\mu>0$), we need to remove $\ker(\tilde{L}_\mu)$.
%Here, the gravitational potential's operator  $\tilde L_\mu$ is defined in \eqref{def-tilde-L-mu}.
If $M' (\mu)\neq0$, then it can be shown that
$\ker\left(\tilde L_\mu\right)=\{\partial_{x^i}V_{\mu}, i=1,2,3\}.$
The three kernel directions  of $\tilde L_\mu$ are removed by a careful choice of the translations $x\mapsto x-y_0$ of the non-rotating stars
such that $ V_{\textup{in}}(x-y_0)-V_{\mu}\perp \partial_{x^i}V_{\mu}$ in $\dot{H}^1(\mathbb{R}^3)$ for $i=1,2,3$.
Consequently, the orbital stability is proved for general perturbations.

Under the spherically symmetric perturbations, we use the estimate
  \eqref{symmetric-p-estimate}
 for nonlinear stability to obtain the uniform energy estimate \eqref{symmetric-p-estimate2}, which leads us to apply a similar approach  as in \cite{CHWY2021,CHWW2023} to  prove the existence of global weak solutions with initial data near the non-rotating stars in Theorem \ref{unconditionalTHM}.
Finally, to drop the continuity assumption of the distance functional on $t$ under the spherically symmetric perturbations, we
first carry out the nonlinear stability analysis to the approximate  solutions which are sufficiently smooth so that the distance functional is continuous on $t$,  and then take limits by using the convexity of the enthalpy to prove the nonlinear stability for the global weak solutions of the Euler-Poisson system.

By Theorem 1.2 in \cite{LZ2022}, the number of growing modes is
\begin{align*}
n^u(\mu)=n^-(L_\mu)-i_\mu,
\end{align*}
where $n^-(L_\mu)$ is the negative dimension of the quadratic form $\langle L_\mu\cdot,\cdot\rangle$ and
\begin{align*}
i_\mu=\begin{cases}
1\quad \text{ if } M'(\mu)\frac{d}{d\mu}\left(\frac{M(\mu)}{R(\mu)}\right)>0\text{ or } M'(\mu)=0, \\
0\quad \text{ if } M'(\mu)\frac{d}{d\mu}\left(\frac{M(\mu)}{R(\mu)}\right)<0\text{ or } \frac{d}{d\mu}\left(\frac{M(\mu)}{R(\mu)}\right)=0.
\end{cases}
\end{align*}

In the case that $M'(\mu)=0$, the index $i_\mu=1$. If $n^-(L_\mu)=1$, then $n^u(\mu)=n^-(L_\mu)-i_\mu=0$. This gives spectral stability of the non-rotating star, but not nonlinear stability in general. A typical  example is  the supermassive stars, where the pressure function satisfies $P(\rho)=\rho^{4/3}$ and the total mass $M$ is independent
on the central density (i.e. $M'(\mu)=0$) \cite{CS1939}. Thus, the supermassive stars are
spectrally stable \cite{LZ2022}.
However, they are proved to be nonlinearly unstable \cite{DLYY2002,CCL2023}
since there exist nearby perturbed solutions with support expanding to infinity when $t\rightarrow\pm\infty$.

Finally, we compare  our  methods and  results  in this
paper with the existing nonlinear stability analysis in the literature.
At first, the nonlinear stability analysis in \cite{RG2003,LTSJ2008} only applies to the non-rotating stars which are global  minimizers of  the energy under the mass constraint. This approach does not work for the non-rotating stars which are  local but not global minimizers.
Our nonlinear stability analysis   in Theorem \ref{main} is directly based on the  sharp linear stability condition $L_\mu|_{Z_{\mu}}\geq0$, where $L_\mu$ is defined in \eqref{def-L-mu} and $Z_\mu$ is the mass-preserving space.
As a corollary of our proof, the linearly stable non-rotating stars are proved to be local minimizers of the energy under the mass constraint.
Secondly, the theory of  nonlinear stability  in \cite{RG2003,LTSJ2008} applies neither to  $\gamma_0<{4/3}$ nor  $\gamma_1<{4/3}$.   Our  results in Theorems \ref{main}-\ref{unconditionalTHM} can be applied to both $\gamma_0\in(6/5,4/3]$ and $\gamma_1\in(6/5,4/3]$.
Thirdly, in \cite{RG2003,LTSJ2008}, the total mass of the perturbed solution is required to be equal to the mass of the steady state. In Theorems \ref{main}-\ref{unconditionalTHM}, we allow the total mass of the perturbed solution to be different from that of the steady state.
Fourthly, nonlinear orbital stability for the   white dwarf stars is proved when the
total mass is beneath a critical mass in \cite{LTSJ2008}.
%For general equations of state, the  existence of spherically symmetric finite-energy solution is  proved when the
%total mass is beneath a critical mass for $\gamma_1\in({6\over5},{4\over3}]$ \cite{CHWW2023}.
It is unknown whether  this critical mass is  optimal.
%It is also unclear  what relation is between the above  critical mass and the  Chandrasekhar limit, which is the supremum of the white dwarf stars' mass, for $\gamma_1\in({6\over5},{4\over3})$.
%For nonlinear stability of the   white dwarf stars, we do not need the assumption that the mass is beneath such a critical mass in Theorem \ref{main}.
Our results confirm the nonlinear stability of any non-rotating white dwarf stars which are known to be linearly stable \cite{LZ2022} and $M'(\mu)\neq0$ \cite{HU2003,SW2019R}.
Fifthly, when $\gamma\in({6\over 5},{4\over3})$,  the existence of global weak solutions with spherical symmetry in \cite{CHWW2023}
 can not be applied to  Theorem \ref{unconditionalTHM} directly, since it is assumed that the
total mass is beneath a critical mass in \cite{CHWW2023}. This assumption is used to deduce a uniform bound of the internal energy $\int\Phi(\rho)dx$ in \cite{CHWW2023}. Without this assumption, our approach is to use the nonlinear stability
 estimate \eqref{symmetric-p-estimate} to obtain the a prior estimate \eqref{symmetric-p-estimate2} for the approximate solutions near the non-rotating stars.
Finally,
we provide a more precise estimate \eqref{distance-t-spherical symmetric perturbations} for the nonlinear stability of the non-rotating stars  under the spherically symmetric perturbations
(i.e. the distance at any time  is of the same order as the initial distance, which is a more precise statement than the usual Lyapunov stability results (e.g., \cite{RG2003,LTSJ2008})).

This paper is organized as follows. In Section 2, we decompose and estimate the energy-Casimir functional in terms of the distance functional. In Section 3, we introduce the dual functional $B$ and  prove its $C^2$ regularity. In Section 4, we study the kernel and negative directions  of $B''(0)$ (i.e. the second order variation of the dual functional $B$ at $0$). We prove the conditional nonlinear orbital stability for general perturbations  in Section 5, while prove the unconditional nonlinear stability for spherically symmetric perturbations in the last section.

\section{\label{HandM}The energy-Casimir and distance functionals}

First, we recall the  definition of a finite-energy weak solution of CEP system \eqref{EP} in \cite{CHWY2021}.
\begin{definition}\label{weaksol}
A measurable vector function $(\rho, v, V)$ is said to be a global finite-energy weak solution of the Cauchy problem \eqref{EP}-\eqref{inaprx} if

(i) $\rho(t,x)\geq0$ a.e., and $v(t,x)=0$ a.e. on the vacuum set $\{(t,x):\rho(t,x)=0\}$.

(ii) For a.e. $t>0$, the total energy is finite:
\begin{align*}
\begin{cases}
\int_{\mathbb{R}^3}\left(\frac{1}{2}\rho v^2+\Phi(\rho)+\frac{1}{8\pi}|\nabla V|^2\right) (t,x)dx\leq C(E_0,M),\\
%\int_{\mathbb{R}^3}\left(\frac{1}{2}\rho v^2+\Phi(\rho)-\frac{1}{8\pi}|\nabla V|^2\right) (t,x)dx\leq \int_{\mathbb{R}^3}\left(\frac{1}{2}\hat{\rho } \hat{v}^2+\Phi(\hat{\rho})-\frac{1}{2}|\nabla \hat{V}|^2\right) (x)dx,
\int_{\mathbb{R}^3}\left(\frac{1}{2}\rho v^2+\Phi(\rho)-\frac{1}{8\pi}|\nabla V|^2\right) (t,x)dx\leq \int_{\mathbb{R}^3}\left(\frac{1}{2}\rho v^2+\Phi(\rho)-\frac{1}{8\pi}|\nabla V|^2\right) (0,x)dx,
\end{cases}
\end{align*}
where $E_0:=\int_{\mathbb{R}^3}\left(\frac{1}{2}(\rho v^2)(0,x)+\Phi(\rho(0,x))\right) dx$.

(iii) For any $\phi(t,x)\in C_0^1([0,\infty)\times\mathbb{R}^3)$,
$$\int_0^\infty\int_{\mathbb{R}^3} (\rho\phi_t+\rho v\cdot \nabla\phi) dxdt+\int_{\mathbb{R}^3}{\rho}(0,x)\phi(0,x)dx=0.$$

(iv) For any $\varphi(t,x)\in (C_0^1([0,\infty)\times\mathbb{R}^3))^3$,
\begin{align*}
\int_0^\infty\int_{\mathbb{R}^3} \left(\rho v\varphi_t+(P(\rho)\text{div}\varphi+\rho v \cdot (v \cdot \nabla)\varphi\right) dxdt\\
+\int_{\mathbb{R}^3}{\rho}(0,x){v}(0,x)\varphi(0,x)dx=\int_0^\infty\int_{\mathbb{R}^3}\rho\nabla V\varphi dxdt.\end{align*}

(v) For any $\xi(x)\in C_0^1(\mathbb{R}^3)$,
$$\int_{\mathbb{R}^3} \nabla V(t,x)\cdot \nabla \xi(x)dx=-4\pi \int_{\mathbb{R}^3}\rho(t,x)\xi(x)dx \quad\; \text{for a.e.}\quad t\geq 0.$$

\end{definition}

Let $V_{\mu}=4\pi\Delta^{-1}\rho_\mu$. It follows from Lemma 3.6 in \cite{LZ2022}  that $V_{\mu}(R_{\mu})=- \frac{M_{\mu}}{R_{\mu}}$.
We define the Energy-Casimir functional
\begin{align}\label{Energy-Casimir functional}
H(\rho,v)=&E(\rho,v)-V_{\mu}(R_{\mu})\int_{\mathbb{R}^3}\rho dx\\\nonumber
=&\int_{\mathbb{R}^3}\left(\frac{1}{2}\rho|v|^2+\Phi(\rho)-\frac{1}{8\pi}|\nabla V|^2\right)dx-V_{\mu}(R_{\mu})\int_{\mathbb{R}^3}\rho dx,\quad(\rho,v)\in X.
\end{align}
Here, we note that $\int_{\mathbb{R}^3}|\nabla V|^2dx<\infty$ for $(\rho,v)\in X$. Indeed,
let
 \begin{align}\label{def-chi1-chi2}
 \chi_1=\chi_{\{{x}\in\mathbb{R}^3|\varepsilon>\rho({x})\geq0\}} \text{ and }\chi_2=\chi_{\{{x}\in\mathbb{R}^3|\rho({x})\geq\varepsilon\}}
 \end{align}
 for fixed $|\varepsilon|\ll1$.
%By H\"{o}lder's inequality[$1=\frac{5}{6}+\frac{1}{6}$], Young's inequality for the convolution[$1+\frac{1}{6}=\frac{5}{6}+\frac{1}{3}$] and
%Sobolev interpolation inequality[$q=5/6,p=1,\gamma\in(6/5,2),a=(5\gamma-6)/(6\gamma-6)>,1-a=\gamma/(6\gamma-6)$], one has
By our assumption \eqref{P2}-\eqref{P3}, $\Phi(\rho)\geq C\rho^{\gamma_{0}}$ for $0<\rho<\varepsilon$, $\varepsilon$ small enough and $\Phi(\rho)\geq C\rho^{\gamma_1}$ for $\rho>N$, $N$ big enough, where $\gamma_{0},\gamma_1\in({6\over5},2)$. Since $\lim_{s\to0^+}{P'(s)\over s^{\gamma_{0}-1}}=K_0>0$,
there exist $\delta_{0}, C_{0}>0$ such that  $\Phi(\rho)={1\over2}{\Phi''(\tilde \rho)}\rho^2={P'(\tilde \rho)\over2\tilde \rho}\rho^2={P'(\tilde \rho)\over2\tilde \rho^{\gamma_{0}-1}}\tilde \rho^{\gamma_{0}-2}\rho^2\geq C_{0} \rho^{\gamma_{0}-2}\rho^{2}\geq\delta_{0}$ for $\varepsilon\leq \rho\leq N$, where $\tilde \rho\in(0,\rho)$. Let $C_1=N^{-\gamma_1}\delta_{0}$. Then $C_1\rho^{\gamma_1}\leq C_1N^{\gamma_1}=\delta_{0}\leq \Phi(\rho)$ for $\varepsilon\leq \rho\leq N$, and thus,
\begin{align}\label{rho-large}
C\rho^{\gamma_1}\leq \Phi(\rho) \;\; \text{ for } \;\; \rho\geq\varepsilon.
\end{align}
Thus,
\begin{align*}
\int_{\mathbb{R}^3}|\nabla V|^2dx=&-4\pi\int_{\mathbb{R}^3}\rho Vdx\leq \|\rho\|_{L^{6/5}(\mathbb{R}^3)}\|V\|_{L^6(\mathbb{R}^3)}\leq C\|\rho\|_{L^{6/5}(\mathbb{R}^3)}^2\\
\leq&C\|\rho\chi_1\|_{L^{6/5}(\mathbb{R}^3)}^2+C\|\rho\chi_2\|_{L^{6/5}(\mathbb{R}^3)}^2\\
\leq& C\|\rho\chi_1\|_{L^1(\mathbb{R}^3)}^{2a_{0}}\|\rho\chi_1\|_{L^{\gamma_{0}}(\mathbb{R}^3)}^{2(1-a_{0})}
+C\|\rho\chi_2\|_{L^1(\mathbb{R}^3)}^{2a_1}\|\rho\chi_2\|_{L^{\gamma_1}(\mathbb{R}^3)}^{2(1-a_1)}\\
\leq &C\left(\int_{\mathbb{R}^3}\rho dx\right)^{2a_{0}}\left(\int_{\mathbb{R}^3}\Phi(\rho )dx\right)^{2(1-a_{0})}+C\left(\int_{\mathbb{R}^3}\rho dx\right)^{2a_{1}}\left(\int_{\mathbb{R}^3}\Phi(\rho )dx\right)^{2(1-a_{1})},
\end{align*}
where $a_i={5\gamma_i-6\over 6\gamma_i-6}$ for $i=0,1$.
Then we divide the difference between the Energy-Casimir functionals $H(\rho,v)$ and $H(\rho_{\mu},0)$ into four terms:
\begin{align}\label{expandCE}
H(\rho,v)-H(\rho_{\mu},0)=&\frac{1}{2}\int_{\mathbb{R}^3}\rho|v|^2dx+\int_{\mathbb{R}^3}(\Phi(\rho)-\Phi(\rho_{\mu})) dx
\\
\notag&-\frac{1}{8\pi}\int_{\mathbb{R}^3}(|\nabla V|^2-|\nabla V_{\mu}|^2)dx-V_{\mu}(R_{\mu})\int_{\mathbb{R}^3}(\rho-\rho_{\mu})dx\\
\notag=&I_1+I_2+I_3+I_4, \quad(\rho,v)\in X.
\end{align}
In order to establish the relation between the  Energy-Casimir terms $H(\rho,v)-H(\rho_{\mu},0)$ and the distance functional $d((\rho,v),(\rho_\mu,0))$ defined in \eqref{distancefun}, we start by estimating $I_2, I_3$ and $I_4$.
Let $B_{\mu}=B(0,R_{\mu})=\{x\in \mathbb{R}^3|\rho_{\mu}(x)>0\}$.  Then by a straightforward computation, we have
\begin{align}\label{I2}
I_2=&\int_{\mathbb{R}^3}(\Phi(\rho)-\Phi(\rho_{\mu})) dx
=\int_{B_{\mu}}(\Phi(\rho)-\Phi(\rho_{\mu})) dx+\int_{B_\mu^c}\Phi(\rho) dx\\\nonumber
=&\int_{B_{\mu}}\left(\Phi(\rho)-\Phi(\rho_{\mu})-\Phi'(\rho_{\mu})(\rho-\rho_{\mu})\right) dx+\int_{B_\mu^c}\Phi(\rho) dx
+\int_{B_{\mu}}\Phi'(\rho_{\mu})(\rho-\rho_{\mu}) dx\\\nonumber
=&I_{21}+I_{22}+I_{23}.
\end{align}
By the Poisson equation $\eqref{EP}_3$, one has
\begin{align}\label{I3}
I_3&=-\frac{1}{8\pi}\int_{\mathbb{R}^3}(|\nabla V|^2-|\nabla V_{\mu}|^2)dx\\\nonumber
&=-\frac{1}{8\pi}\int_{\mathbb{R}^3}|\nabla V-\nabla V_{\mu}|^2dx-\frac{1}{4\pi}\int_{\mathbb{R}^3}\nabla V_{\mu}\cdot\left(\nabla V-\nabla V_{\mu}\right) dx\\\nonumber
&=-\frac{1}{8\pi}\int_{\mathbb{R}^3}|\nabla V-\nabla V_{\mu}|^2dx+\int_{\mathbb{R}^3} V_{\mu}(\rho-\rho_{\mu}) dx\\\nonumber
&=-\frac{1}{8\pi}\int_{\mathbb{R}^3}|\nabla V-\nabla V_{\mu}|^2dx+\int_{B_\mu^c} V_{\mu}\rho dx+\int_{B_{\mu}} V_{\mu}(\rho-\rho_{\mu}) dx\\\nonumber
&=I_{31}+I_{32}+I_{33}.
\end{align}
Moreover,
\begin{align}\label{I4}
I_4&=-V_{\mu}(R_{\mu})\int_{B_\mu^c}\rho dx-V_{\mu}(R_{\mu})\int_{B_{\mu}}\left(\rho-\rho_{\mu}\right)dx=I_{41}+I_{42}.
\end{align}
The non-rotating star $(\rho_\mu,0)$  satisfies the following equation
$$
\nabla p(\rho_{\mu})=-\rho_{\mu}\nabla V_{\mu},
$$
which gives
$$
\Phi'(\rho_{\mu})=- V_{\mu}+V_{\mu}(R_{\mu}).
$$
Thus,
\begin{align}\label{I23-33-42}
I_{23}+I_{33}+I_{42}=\int_{B_{\mu}}\left(\Phi'(\rho_{\mu})+ V_{\mu}-V_{\mu}(R_{\mu})\right)(\rho-\rho_{\mu}) dx=0.
\end{align}
Combining \eqref{expandCE}-\eqref{I23-33-42}, we have
\begin{align}\label{Hrh0ovrhomu02}
&H(\rho,v)-H(\rho_{\mu},0)=I_1+I_{21}+I_{22}+I_{31}+I_{32}+I_{41}\\\nonumber
=&\frac{1}{2}\int_{\mathbb{R}^3}\rho|v|^2dx+\int_{B_{\mu}}\left(\Phi(\rho)-\Phi(\rho_{\mu})-\Phi'(\rho_{\mu})(\rho-\rho_{\mu})\right) dx\\\nonumber
&+\int_{B_\mu^c}\Phi(\rho) dx-\frac{1}{8\pi}\int_{\mathbb{R}^3}|\nabla V-\nabla V_{\mu}|^2dx+\int_{B_\mu^c} (V_{\mu}-V_{\mu}(R_{\mu}))\rho dx.
\end{align}
Let $\rho_{\text{in}}=\rho\chi_{B_{\mu}}$, $\rho_{\text{out}}=\rho-\rho_{\text{in}}$, $\Delta V_{\text{in}}=4\pi \rho_{\text{in}}$  and $\Delta V_{\text{out}}=4\pi \rho_{\text{out}}$, where $\chi_{B_\mu}$ is the indicator function for $B_\mu$. For any perturbed density $\rho$ near the steady state $(\rho_{\mu}(x+x_0),0)$, we define
$\rho_{\text{in}}=\rho\chi_{B_{\mu}-x_0}$ without causing confusion, where  $B_{\mu}-x_0=\{x|x=y-x_0,y\in B_{\mu}\}$ .
Then
\begin{align}\label{I31}
&I_{31}=
-\frac{1}{8\pi}\int_{\mathbb{R}^3}|\nabla V-\nabla V_{\mu}|^2dx\\\nonumber
=&-\frac{1}{8\pi}\int_{\mathbb{R}^3}|\nabla V_{\text{in}}+\nabla V_{\text{out}}-\nabla V_{\mu}|^2dx\\\nonumber
=&-\frac{1}{8\pi}\int_{\mathbb{R}^3}|\nabla V_{\text{in}}-\nabla V_{\mu}|^2dx-\frac{1}{8\pi}\int_{\mathbb{R}^3}|\nabla V_{\text{out}}|^2dx-\frac{1}{4\pi}\int_{\mathbb{R}^3}\nabla V_{\text{out}}\cdot(\nabla V_{\text{in}}-\nabla V_{\mu})dx.
\end{align}
Inserting \eqref{I31} into \eqref{Hrh0ovrhomu02}, we have
\begin{align*}
H(\rho,v)-H(\rho_{\mu},0)
&=\frac{1}{2}\int_{\mathbb{R}^3}\rho|v|^2dx+\int_{B_{\mu}}\left(\Phi(\rho)-\Phi(\rho_{\mu})-\Phi'(\rho_{\mu})(\rho-\rho_{\mu})\right) dx+\int_{B_\mu^c}\Phi(\rho) dx\\
&\quad-\frac{1}{8\pi}\int_{\mathbb{R}^3}|\nabla V_{\text{in}}-\nabla V_{\mu}|^2dx+\int_{B_\mu^c} (V_{\mu}-V_{\mu}(R_{\mu}))\rho dx\\
&\quad
-\frac{1}{8\pi}\int_{\mathbb{R}^3}|\nabla V_{\text{out}}|^2dx
-\frac{1}{4\pi}\int_{\mathbb{R}^3}\nabla V_{\text{out}}\cdot(\nabla V_{\text{in}}-\nabla V_{\mu})dx.
\end{align*}

Motivated by the above expression, we define
the distance between $(\rho,v)\in X$ and $(\rho_\mu,0)$ by
\begin{align}\label{distancefun}
d((\rho,v),(\rho_\mu,0))&=\frac{1}{2}\int_{\mathbb{R}^3}\rho|v|^2dx
+\int_{B_\mu}\left(\Phi(\rho)-\Phi(\rho_\mu)-\Phi'(\rho_\mu)(\rho-\rho_\mu)\right) dx\\
\notag&\quad+\int_{B_\mu^c}\Phi(\rho) dx+\frac{1}{8\pi}\int_{\mathbb{R}^3}|\nabla V_{\text{in}}-\nabla V_\mu|^2dx+\int_{B_\mu^c} (V_\mu-V_\mu(R_{\mu}))\rho dx.
\end{align}
Now, we check that  the distance functional \eqref{distancefun} is well-defined for $(\rho,v)\in X$, where $X$ is defined in \eqref{def-X-mu}.

\begin{lemma}\label{distance-well defined}
The distance $d((\rho,v),(\rho_\mu,0))$ in  \eqref{distancefun} is  well-defined for $(\rho,v)\in X$.
\end{lemma}
\begin{proof}
Let $(\rho,v)\in X$.
We divide $d((\rho,v),(\rho_\mu,0))$ into five terms:
\begin{align}\label{d1-5}
\begin{cases}
d_1((\rho,v),(\rho_\mu,0))=\frac{1}{2}\int_{\mathbb{R}^3}\rho|v|^2dx,\\
d_2((\rho,v),(\rho_\mu,0))=\int_{B_{\mu}}\left(\Phi(\rho)-\Phi(\rho_{\mu})-\Phi'(\rho_{\mu})(\rho-\rho_{\mu})\right) dx,\\
d_3((\rho,v),(\rho_\mu,0))=\int_{B_\mu^c}\Phi(\rho) dx,\\
d_4((\rho,v),(\rho_\mu,0))=\frac{1}{8\pi}\int_{\mathbb{R}^3}|\nabla V_{\text{in}}-\nabla V_{\mu}|^2dx,\\
d_5((\rho,v),(\rho_\mu,0))=\int_{B_\mu^c} (V_{\mu}-V_{\mu}(R_{\mu}))\rho dx.\\
\end{cases}
\end{align}
Every term in \eqref{d1-5} is finite since $(\rho,v)\in X$. It suffices to show that $d_i((\rho,v),(\rho_\mu,0))\geq0$ for $i=2,5$.
The fact that $\Phi''\geq0$ implies that
\begin{align*}
d_2((\rho,v),(\rho_\mu,0))=\int_{B_{\mu}}\left(\Phi(\rho)-\Phi(\rho_{\mu})-\Phi'(\rho_{\mu})(\rho-\rho_{\mu})\right) dx=\int_{B_{\mu}}\Phi''(\xi_{\rho,\rho_\mu})(\rho-\rho_{\mu})^2 dx\geq0,
\end{align*}
where $\xi_{\rho,\rho_\mu}$ is between $\rho_{\mu}$ and $\rho$.
Since $V_{\mu}(r)=-\frac{M_{\mu}}{r}$ for $r\geq R_{\mu}$ by Lemma 3.6 in \cite{LZ2022},  we have
\begin{align*}
d_5((\rho,v),(\rho_\mu,0))&=\int_{B_\mu^c} (V_{\mu}-V_{\mu}(R_{\mu}))\rho dx
=\int_{B_\mu^c} \left(-\frac{M_{\mu}}{r}+\frac{M_{\mu}}{R_{\mu}}\right)\rho dx\geq0.
\end{align*}
\end{proof}
For the sake of convenience, we denote $d=d((\rho,v),(\rho_\mu,0))$ and $d_i=d_i((\rho,v),(\rho_\mu,0))$ for $i=1,\cdots, 5$.
For $d_5$, we need the following lower bound.
\begin{lemma}
Let $(\rho,v)\in X$. Then for any $\delta>0$,
\begin{align}\label{out-c}
\frac{\delta M_{\mu}}{(R_{\mu}+\delta)R_{\mu}}\int_{ B(0,R_{\mu}+\delta)^c}\rho_{\rm{out}}dx\leq d_5.
\end{align}
\end{lemma}
\begin{proof}
Again by Lemma 3.6 in \cite{LZ2022}, we have
\begin{align*}
d_5=&\left(\int_{B_\mu^c\cap B(0,R_{\mu}+\delta)}+\int_{ B(0,R_{\mu}+\delta)^c} \right) \left(-\frac{M_{\mu}}{r}+\frac{M_{\mu}}{R_{\mu}}\right)\rho_{\text{out}} dx\\
\geq&\int_{B_\mu^c\cap B(0,R_{\mu}+\delta)}\left(-\frac{M_{\mu}}{r}+\frac{M_{\mu}}{R_{\mu}}\right)\rho_{\text{out}} dx+\left(-\frac{M_{\mu}}{R_{\mu}+\delta}+\frac{M_{\mu}}{R_{\mu}}\right)\int_{ B(0,R_{\mu}+\delta)^c}\rho_{\text{out}} dx\\
\geq&\frac{\delta M_{\mu}}{(R_{\mu}+\delta)R_{\mu}}\int_{ B(0,R_{\mu}+\delta)^c}\rho_{\text{out}}dx.
\end{align*}
\end{proof}
Thus, we have the relation between $H$ and $d$:
\begin{align*}
H(\rho,v)-H(\rho_{\mu},0)=&d_1+d_2+d_3
-d_4+d_5\\\nonumber
&-\frac{1}{8\pi}\int_{\mathbb{R}^3}|\nabla V_{\text{out}}|^2dx
-\frac{1}{4\pi}\int_{\mathbb{R}^3}\nabla V_{\text{out}}\cdot(\nabla V_{\text{in}}-\nabla V_{\mu})dx.
\end{align*}
We show that the last two terms are high order terms of the distance function. More precisely, we have the following result.
\begin{lemma}\label{high order term EC functional}
\begin{align}\label{firsthighorder}
\frac{1}{8\pi}\int_{\mathbb{R}^3}|\nabla V_{\textup{out}}|^2dx&\leq C(d^{\frac{2}{\gamma_{0}}}+d^{\frac{2}{\gamma_1}}+d^{5\over3}),\\\label{secondhighorder}
\frac{1}{4\pi}\int_{\mathbb{R}^3}\nabla V_{\textup{out}}\cdot(\nabla V_{\textup{in}}-\nabla V_{\mu})dx
&\leq Cd^{\frac{1}{2}}\left(d^{\frac{2}{\gamma_{0}}}+d^{\frac{2}{\gamma_1}}+d^{5\over3}\right)^{1\over2}.
\end{align}
\end{lemma}
\begin{proof}
Denote
\begin{align}\label{denoteout}
\begin{cases}
\rho^1_{\text{out}}=\rho_{\text{out}}\chi_{B_\mu^c\cap B(0,R_{\mu}+\delta)},\\
\rho^2_{\text{out}}=\rho_{\text{out}}\chi_{ B_\mu^c(0,R_{\mu}+\delta)},\\
\Delta V_{\text{out}}^i=4\pi \rho_{\text{out}}^i \text{ for } i=1,2.
\end{cases}
\end{align}
where $\chi_\Omega$ denote the indicator function for $\Omega$. Then $\rho_{\text{out}}=\rho^1_{\text{out}}+\rho^2_{\text{out}}$.
Now we begin to separate the first  term into two parts:
\begin{align}\label{rho-out}
\frac{1}{8\pi}\int_{\mathbb{R}^3}|\nabla V_{\text{out}}|^2dx&\leq C\int_{\mathbb{R}^3}\left(|\nabla V^1_{\text{out}}|^2+|\nabla V^2_{\text{out}}|^2\right)dx.
\end{align}
By integration by parts, $\eqref{denoteout}_3$, and Young's inequality, we obtain
\begin{align}\nonumber
\frac{1}{8\pi}\int_{\mathbb{R}^3}|\nabla V^1_{\text{out}}|^2dx&=-\frac{1}{8\pi}\int_{\mathbb{R}^3} V^1_{\text{out}}\Delta V^1_{\text{out}}dx\\\nonumber
&=-{1\over2}\int_{\mathbb{R}^3} V^1_{\text{out}}\rho^1_{\text{out}}dx\\\nonumber
&\leq {1\over2}\|\rho^1_{\text{out}}\|_{L^{6/5}(\mathbb{R}^3)}\|V^1_{\text{out}}\|_{L^{6}(\mathbb{R}^3)}\\\nonumber
&\leq C\|\rho^1_{\text{out}}\|_{L^{6/5}(\mathbb{R}^3)}^2\leq C
\|\rho^1_{\text{out}}\chi_1\|_{L^{6/5}(\tilde B)}^2+C\|\rho^1_{\text{out}}\chi_2\|_{L^{6/5}(\tilde B)}^2\\\nonumber
&\leq C\|\rho^1_{\text{out}}\chi_1\|_{L^{\gamma_{0}}(\tilde B)}^2+C\|\rho^1_{\text{out}}\chi_2\|_{L^{\gamma_1}(\tilde B)}^2\\
&\leq Cd_3^{2\over\gamma_{0}}+Cd_3^{2\over\gamma_1} \leq C(d^{2\over\gamma_{0}}+d^{2\over\gamma_1}),
\label{rho-out-1}
\end{align}
where $\chi_1$ and $\chi_2$ are defined in \eqref{def-chi1-chi2} and $\tilde B=B_\mu^c\cap B(0,R_{\mu}+\delta)$. We use the fact that $\gamma_0,\gamma_1>6/5$
and the H\"{o}lder inequality $\|u\|_{L^p(\Omega)}\leq \|u\|_{L^q(\Omega)}|\Omega|^{{q-p\over pq}}$
 for $u\in L^q(\Omega)$, where $\Omega$ is  a bounded domain and $1\leq p\leq q\leq \infty$.

By Sobolev interpolation inequality and  \eqref{out-c}, we have
\begin{align}\nonumber
\frac{1}{8\pi}\int_{\mathbb{R}^3}|\nabla V_{\text{out}}^2|^2dx
\leq& C\|\rho^2_{\text{out}}\|_{L^{6/5}(\mathbb{R}^3)}^2\leq C
\|\rho^2_{\text{out}}\chi_1\|_{L^{6/5}(\mathbb{R}^3)}^2+C\|\rho^2_{\text{out}}\chi_2\|_{L^{6/5}(\mathbb{R}^3)}^2\\\nonumber
\leq& C\|\rho_{\text{out}}^2\chi_1\|_{L^1(\mathbb{R}^3)}^{2a_{0}}\|\rho_{\text{out}}^2\chi_1\|_{L^{\gamma_{0}}(\mathbb{R}^3)}^{2(1-a_{0})}
+C\|\rho_{\text{out}}^2\chi_2\|_{L^1(\mathbb{R}^3)}^{2a_1}\|\rho_{\text{out}}^2\chi_2\|_{L^{\gamma_1}(\mathbb{R}^3)}^{2(1-a_1)}\\
\leq& C d_5^{2a_{0}} d_3^{2(1-a_{0})\over \gamma_{0}}+
Cd_5^{2a_1} d_3^{2(1-a_1)\over \gamma_1}\leq C d^{5\over3},
\label{rho-out-2}
\end{align}
where $a_i={5\gamma_i-6\over 6\gamma_i-6}$ for $i=0,1$.
Combining \eqref{rho-out}, \eqref{rho-out-1} and \eqref{rho-out-2}, we obtain
\eqref{firsthighorder}.
 We have
\begin{align*}
\notag\frac{1}{4\pi}\int_{\mathbb{R}^3}\nabla V_{\text{out}}\cdot(\nabla V_{\text{in}}-\nabla V_{\mu})dx&\leq C\|\nabla V_{\text{in}}-\nabla V_{\mu}\|_{L^2}\|\nabla V_{\text{out}}\|_{L^2}\\
&\leq Cd^{\frac{1}{2}}\left(d^{\frac{2}{\gamma_{0}}}+d^{\frac{2}{\gamma_1}}+d^{5\over3}\right)^{1\over2}.
\end{align*}
This proves \eqref{secondhighorder}.
\end{proof}
Combining \eqref{firsthighorder} and \eqref{secondhighorder}, we can now rewrite \eqref{H-rho-vHrhomu0} as
\begin{align}
H(\rho,v)-H(\rho_{\mu},0)\geq& d_1+d_3+d_5
+\bigg(d_2-d_4\bigg)-o(d).\label{H-rho-H-rho0}
\end{align}
Finally, we pay our attention to $d_2-d_4$. Indeed,
\begin{align}\nonumber
d_2-d_4
=&\int_{B_\mu}\left(\Phi(\rho)-\Phi(\rho_\mu)-\Phi'(\rho_\mu)(\rho-\rho_\mu)\right) dx-\frac{1}{8\pi}\int_{\mathbb{R}^3}|\nabla V_{\text{in}}-\nabla V_\mu|^2dx\\\nonumber
=&\int_{B_\mu}\left(\Phi(\rho)-\Phi(\rho_\mu)-\Phi'(\rho_\mu)(\rho-\rho_\mu)\right) dx+\frac{1}{8\pi}\int_{\mathbb{R}^3}|\nabla V_{\text{in}}-\nabla V_\mu|^2dx\\\label{tilde-I}
&\quad+\int_{B_\mu}(V_{\text{in}}-V_\mu)(\rho_{\text{in}}-\rho_\mu)dx.
\end{align}
Let $\tilde{V}_{\text{in}}=V_{\text{in}}-V_\mu$ and $\tilde{\rho}_{\text{in}}=\rho_{\text{in}}-\rho_\mu$.
\if0
Noting that $\Phi''(\rho_0)={P'(\rho_0)\over \rho_0}>0$ for $0\leq r<R$ and $\lim_{r\to R_0^-}\Phi''(\rho_0)=\infty$, there exists $\delta_0>0$ such that
$\Phi''(\rho_0)\geq \delta_0>0$ for $0\leq r\leq R_0$.
Thus,
\begin{align}\label{M(0)}
|M(0)-M_0|\leq \|\rho-\rho_0\|_{
\end{align}
\fi
Then by  our assumption that $\Phi(\rho)\geq C\rho^{\gamma_0}$ for $0<\rho<\varepsilon$, \eqref{out-c} and \eqref{rho-large}, we have
\begin{align}\nonumber
&\left|\int_{B_{\mu}}\tilde{\rho}_{\text{in}}dx\right|
=\left|\int_{\mathbb{R}^3}\left(\rho-\rho_{\mu}\right)dx-\int_{B_{\mu}^c}\rho_{\text{out}}dx\right|\\\nonumber
\leq& |M-M_{\mu}|+\|\rho_{\text{out}}^1\|_{L^1(\tilde B)}+\int_{ B(0,R_{\mu}+\delta)^c} \rho_{\text{out}}dx\\\nonumber
\leq&|M-M_{\mu}|+\|\rho_{\text{out}}^1\chi_1\|_{L^1(\tilde B)}
+\|\rho_{\text{out}}^1\chi_2\|_{L^1(\tilde B)}+Cd_5\\\nonumber
\leq&|M-M_{\mu}|+C\|\rho_{\text{out}}^1\chi_1\|_{L^{\gamma_{0}}(\tilde B)}
+C\|\rho_{\text{out}}^1\chi_2\|_{L^{\gamma_1}(\tilde B)}+Cd_5\\\nonumber
\leq &|M-M_{\mu}|+Cd_3^{1\over\gamma_{0}}
+Cd_3^{1\over\gamma_1}+Cd_5\\
\leq &|M-M_{\mu}|+C(d^{1\over\gamma_{0}}
+d^{1\over\gamma_1}+d),\label{tilde-rho-in}
\end{align}
where we still denote $M=\int_{\mathbb{R}^3}\rho dx$ in this section without causing confusion, and  $\tilde B=B_\mu^c\cap B(0,R_{\mu}+\delta)$.
We define
the projection $P$ by
\begin{align}\label{ProjectionP}
P\phi=P^\mu\phi:=\frac{\int_{B_{\mu}}\frac{\phi}{\Phi''(\rho_{\mu})}dx}{\int_{B_{\mu}}\frac{1}{\Phi''(\rho_{\mu})}dx}
\end{align}
for $\phi\in L^1_{1\over{\Phi''(\rho_{\mu})}}(B_{\mu})$.
%then $\int_{\mathbb{R}^3}\frac{1}{\Phi''(\rho_0)}(\phi-P\phi) dxdydz=0$.
Note that ${1\over\Phi''(\rho_{\mu})}$ is compacted supported in $B_{\mu}$ since $$\lim\limits_{r\to R_{\mu}^-}{1\over\Phi''(\rho_{\mu}(r))}=\lim\limits_{r\to R_{\mu}^-}{\rho_{\mu}(r)^{\gamma_{0}-1}\over P'(\rho_{\mu}(r))}\rho_{\mu}(r)^{2-\gamma_{0}}=0,$$ where we used \eqref{P2}.
Then
\begin{align}\nonumber
|P\tilde{V}_{\text{in}}|&=\left|\frac{\int_{B_{\mu}}\frac{\tilde{V}_{\text{in}}}{\Phi''(\rho_{\mu})}dx}{\int_{B_{\mu}}\frac{1}{\Phi''(\rho_{\mu})}dx}\right|
\leq \frac{\sup\limits_{x\in B_{\mu}}\frac{1}{\Phi''(\rho_{\mu})}}{\int_{B_{\mu}}\frac{1}{\Phi''(\rho_{\mu})}dx}\int_{B_{\mu}}|\tilde{V}_{\text{in}}|dx\leq C \|\tilde{V}_{\text{in}}\|_{L^1({B_{\mu}})}\\\label{P}
&\leq C \|\tilde{V}_{\text{in}}\|_{L^6({B_{\mu}})}\leq C \|\tilde{V}_{\text{in}}\|_{L^6(\mathbb{R}^3)}\leq C \|\nabla \tilde{V}_{\text{in}}\|_{L^2(\mathbb{R}^3)}\leq C d_4^{\frac{1}{2}}.
\end{align}
By \eqref{tilde-rho-in} and \eqref{P}, we have
\begin{align}\label{P-2}
\left|P\tilde{V}_{\text{in}}\int_{B_{\mu}}\tilde{\rho}_{\text{in}}dx\right|\leq C d^{\frac{1}{2}}\left(|M-M_{\mu}|+d^{1\over\gamma_{0}}
+d^{1\over\gamma_1}+d\right).
\end{align}
By \eqref{tilde-I} and \eqref{P-2}, we have
\begin{align}\label{d2-d4-estimate}
&d_2-d_4\\\nonumber
=& {1\over 8\pi}\int_{\mathbb{R}^3}|\nabla\tilde V_{\text{in}}|^2dx+\int_{B_{\mu}}\bigg(\left(\Phi(\tilde\rho_{\text{in}}+\rho_{\mu})-\Phi(\rho_{\mu})-\Phi'(\rho_{\mu})\tilde\rho_{\text{in}}\right)-(P\tilde{V}_{\text{in}}-\tilde{V}_{\text{in}})\tilde{\rho}_{\text{in}} \bigg)dx \\\nonumber &
 +P\tilde V_{\text{in}}\int_{B_{\mu}}
\tilde \rho_{\text{in}}dx\\\nonumber
\geq&{1\over 8\pi}\int_{\mathbb{R}^3}|\nabla\tilde V_{\text{in}}|^2dx+\int_{B_{\mu}}\bigg(\left(\Phi(\tilde\rho_{\text{in}}+\rho_{\mu})-\Phi(\rho_{\mu})-\Phi'(\rho_{\mu})\tilde\rho_{\text{in}}\right)-(P\tilde{V}_{\text{in}}-\tilde{V}_{\text{in}})\tilde{\rho}_{\text{in}} \bigg)dx \\\nonumber &
 -C d^{\frac{1}{2}}\left(|M-M_{\mu}|+d^{1\over\gamma_{0}}
+d^{1\over\gamma_1}+d\right).
\end{align}
\section{The dual functional and its regularity}

To deal with the term $$\int_{B_{\mu}}\left(\left(\Phi(\tilde\rho_{\text{in}}+\rho_{\mu})-\Phi(\rho_{\mu})-\Phi'(\rho_{\mu})\tilde\rho_{\text{in}}\right)-(P\tilde{V}_{\text{in}}-\tilde{V}_{\text{in}})\tilde{\rho}_{\text{in}} \right)dx$$ in \eqref{d2-d4-estimate},
we intend to study its Taylor expansion near $(\rho_\mu,0)$. However, this functional is not regular enough (i.e. $C^2$) to justify the Taylor expansion.  To overcome this difficulty,  we instead bound it from below by  its dual functional and study the Taylor
expansion of the dual functional, which is in $C^2$.

We denote \begin{align*}
\Psi_{\rho_{\mu}(r)}(\tau)=\Phi(\tau+\rho_{\mu}(r))-\Phi(\rho_{\mu}(r))-\Phi'(\rho_{\mu}(r))\tau \text{  and }H_{\rho_{\mu}(r),y}(\tau)=\Psi_{\rho_{\mu}(r)}(\tau)-y\tau
\end{align*}
for $\tau\geq -\rho_{\mu}(r)$,
where  $r\in[0,R_\mu]$ and
$y\in\mathbb{R}$.
Then the Legendre transformation of $\Psi_{\rho_{\mu}(r)}$ is given by
\begin{align*}
\Psi^*_{\rho_{\mu}(r)}(y)=\inf_{\tau\geq-\rho_{\mu}(r)}H_{\rho_{\mu}(r),y}(\tau),\quad y\in\mathbb{R}.
\end{align*}
Now, we study the regularity of $\Psi^*_{\rho_{\mu}(r)}$.
\begin{lemma}$\Psi^*_{\rho_{\mu}(r)} \in C^2(\mathbb{R})$ and
\begin{align}\nonumber
(\Psi^*_{\rho_{\mu}(r)})'(y)&=-(\Phi')_+^{-1}\left(y+\Phi'(\rho_{\mu}(r))\right)+\rho_{\mu}(r),\\\label{z(y)} (\Psi^*_{\rho_{\mu}(r)})''(y)&=-{1\over\Phi''\left((\Phi')_+^{-1}\left(y+\Phi'(\rho_{\mu}(r))\right)\right)},
\end{align}
for $y\in \mathbb{R}$, where $0\leq r\leq R_{\mu}$.
In particular,
\begin{align*}
\Psi^*_{\rho_{\mu}(r)}(0)=(\Psi^*_{\rho_{\mu}(r)})'(0)=0 \text{ and }(\Psi^*_{\rho_{\mu}(r)})''(0)=-{1\over \Phi''(\rho_{\mu}(r))},
\end{align*}
where $0\leq r\leq R_{\mu}.$
\end{lemma}
\begin{proof}
Direct computation gives
$$H_{\rho_{\mu}(r),y}'(\tau)=\Phi'(\tau+\rho_{\mu}(r))-\Phi'(\rho_{\mu}(r))-y,$$
and
$$H_{\rho_{\mu}(r),y}''(\tau)=\Phi''(\tau+\rho_{\mu}(r))={P'(\tau+\rho_{\mu}(r))\over \tau+\rho_{\mu}(r)}>0,$$
for $\tau\geq -\rho_{\mu}(r).$
In particular,
$$H_{\rho_{\mu}(r),y}'(-\rho_{\mu}(r))=\Phi'(0)-\Phi'(\rho_{\mu}(r))-y=-\Phi'(\rho_{\mu}(r))-y.$$
Thus,  if $y> -\Phi'(\rho_{\mu}(r))$, then $H_{\rho_{\mu}(r),y}'(-\rho_{\mu}(r))<0$ and there exists a unique $z(y)=z_{\rho_{\mu}(r)}(y)\in(-\rho_{\mu}(r),\infty)$ such that
\begin{align*}
H_{\rho_{\mu}(r),y}'(z(y))&=\Phi'(z(y)+\rho_{\mu}(r))-\Phi'(\rho_{\mu}(r))-y=0.
\end{align*}
This implies that
\begin{align}
\Psi^*_{\rho_{\mu}(r)}(y)&=H_{\rho_{\mu}(r),y}(z(y)).\label{H-rho-0(r)y(x)}
\end{align}
Here, $z(y)=z_{\rho_{\mu}(r)}(y)$ is in fact dependent on $\rho_{\mu}(r)$, and  below we always use $z(y)$ to avoid tedious notation without confusion.
In this case, since $\Phi''(z(y)+\rho_{\mu}(r))>0$ for $y> -\Phi'(\rho_{\mu}(r))$,  by the Implicit Function Theorem we have $ z\in C^1(-\Phi'(\rho_{\mu}(r),\infty)$ and
\begin{align}\label{z(y)+}
z(y)=(\Phi')^{-1}\left(y+\Phi'(\rho_{\mu}(r))\right)-\rho_{\mu}(r) \text{ and } z'(y)={1\over\Phi''\left(z(y)+\rho_{\mu}(r)\right)}.
\end{align}
If $y\leq -\Phi'(\rho_{\mu}(r))$, then $H_{\rho_{\mu}(r),y}'(-\rho_{\mu}(r))\geq0$ and
\begin{align}\label{z(y)-}z(y)=-\rho_{\mu}(r)\end{align}
 is the unique point in $[-\rho_{\mu}(r),\infty)$ such that
\begin{align}\label{Psi*-rho-0(r)-}
\Psi^*_{\rho_{\mu}(r)}(y)=H_{\rho_{\mu}(r),y}(z(y))=-\Phi(\rho_{\mu}(r))+\Phi'(\rho_{\mu}(r))\rho_{\mu}(r)+y\rho_{\mu}(r).
\end{align}
Then $z'(y)\equiv0$ for $y\in(-\infty,-\Phi'(\rho_{\mu}(r)]$.  To prove that $z\in C^1(\mathbb{R})$, it suffices to show that $z$ is differentiable at $y=-\Phi'(\rho_{\mu}(r))$. In fact, by \eqref{z(y)+} we have
$$\lim_{y\to-\Phi'(\rho_{\mu}(r))^+} z(y)=-\rho_{\mu}(r)=z(-\Phi'(\rho_{\mu}(r))),$$
and
\begin{align*}
\lim_{y\to-\Phi'(\rho_{\mu}(r))^+}z'(y)=\lim_{y\to-\Phi'(\rho_{\mu}(r))^+}{1\over\Phi''\left(z(y)+\rho_{\mu}(r)\right)}=0.
\end{align*}
Then we compute $(\Psi^*_{\rho_{\mu}(r)})'(y)$ and $(\Psi^*_{\rho_{\mu}(r)})''(y)$. For $y> -\Phi'(\rho_{\mu}(r))$, by
\eqref{H-rho-0(r)y(x)} and \eqref{z(y)+} we have
\begin{align}\nonumber
&(\Psi^*_{\rho_{\mu}(r)})'(y)=\partial_y(H_{\rho_{\mu}(r),y}(z(y)))\\\nonumber
=&\partial_y\left(\Phi(z(y)+\rho_{\mu}(r))-\Phi(\rho_{\mu}(r))-\Phi'(\rho_{\mu}(r))z(y)-yz(y)\right)\\\nonumber
=&\Phi'(z(y)+\rho_{\mu}(r))z'(y)-\Phi'(\rho_{\mu}(r))z'(y)-z(y)-yz'(y)\\\nonumber
=&\Phi'(z(y)+\rho_{\mu}(r))z'(y)-\Phi'(\rho_{\mu}(r))z'(y)-z(y)-\left(\Phi'(z(y)+\rho_{\mu}(r))-\Phi'(\rho_{\mu}(r))\right)z'(y)\\\nonumber
=&-z(y)=-(\Phi')^{-1}\left(y+\Phi'(\rho_{\mu}(r))\right)+\rho_{\mu}(r),
\end{align}
and
\begin{align}\label{Psi*derivative+}
&(\Psi^*_{\rho_{\mu}(r)})''(y)=-z'(y)=-{1\over\Phi''\left(z(y)+\rho_{\mu}(r)\right)}.
\end{align}
For $y\leq -\Phi'(\rho_{\mu}(r))$, by \eqref{Psi*-rho-0(r)-} and \eqref{z(y)-} we also have
\begin{align*}
(\Psi^*_{\rho_{\mu}(r)})'(y)&=\partial_y(-\Phi(\rho_{\mu}(r))+\Phi'(\rho_{\mu}(r))\rho_{\mu}(r)+y\rho_{\mu}(r))\\
&=\rho_{\mu}(r)=-z(y)=\rho_{\mu}(r),
\end{align*}
and
\begin{align}\label{Psi*derivative-}
(\Psi^*_{\rho_{\mu}(r)})''(y)&=-z'(y)=0.
\end{align}
If $0\leq r<R_{\mu}$, by \eqref{Psi*derivative+} we have 
$$(\Psi^*_{\rho_{\mu}(r)})'(0)=-z(0)=0, (\Psi^*_{\rho_{\mu}(r)})''(0)=-z'(0)=-{1\over\Phi''\left(\rho_{\mu}(r)\right)}.$$
In this case, $\Psi^*_{\rho_{\mu}(r)}(0)=0$ by \eqref{H-rho-0(r)y(x)}. If $r=R_{\mu}$, by \eqref{Psi*derivative-}, \eqref{Psi*-rho-0(r)-} and $z\in C^1(\mathbb{R})$ we have  $(\Psi^*_{\rho_{\mu}(r)})'(0)=-z(0)=\rho_{\mu}(R_{\mu})=0$, $(\Psi^*_{\rho_{\mu}(r)})''(0)=-z'(0)=0$ and $\Psi^*_{\rho_{\mu}(r)}(0)=0$. Thus,
\begin{align*}
\Psi^*_{\rho_{\mu}(r)}(0)=(\Psi^*_{\rho_{\mu}(r)})'(0)=0 \text{ and }(\Psi^*_{\rho_{\mu}(r)})''(0)=-{1\over \Phi''(\rho_{\mu}(r))}
\end{align*}
for $0\leq r\leq R_{\mu}.$

We extend $(\Phi')^{-1}$ to $s\in(-\infty,0)$ by zero function and denote the extended function by $(\Phi')^{-1}_+:\mathbb{R}\to[0,\infty)$. Then $(\Phi')^{-1}_+\in C^1(\mathbb{R})$ and $(\Phi')^{-1}_+$ is increasing on $[0,\infty)$. By \eqref{Psi*derivative+} and \eqref{Psi*derivative-}, we have for $r\in[0,R_{\mu}]$ and $y\in\mathbb{R}$,
\begin{align}\nonumber
(\Psi^*_{\rho_{\mu}(r)})'(y)&=-z(y)=-(\Phi')_+^{-1}\left(y+\Phi'(\rho_{\mu}(r))\right)+\rho_{\mu}(r),\\\nonumber (\Psi^*_{\rho_{\mu}(r)})''(y)&=-z'(y)=-{1\over\Phi''\left(z(y)+\rho_{\mu}(r)\right)}=-{1\over\Phi''\left((\Phi')_+^{-1}\left(y+\Phi'(\rho_{\mu}(r))\right)\right)}.
\end{align}
Since $z\in C^1(\mathbb{R})$, we have $\Psi^*_{\rho_{\mu}(r)}\in C^2(\mathbb{R})$ for $r\in[0,R_{\mu}]$.
\end{proof}

 We define
the functionals $A$ and $B$ by
\begin{align*}
A(\phi)=A^\mu(\phi):=\int_{B_{\mu}}\Psi^*_{\rho_{\mu}}(P\phi-\phi) dx,\quad\phi\in \dot{H}^1(\mathbb{R}^3),
\end{align*}
and
\begin{align}\label{dual functional expression}
B(\phi)=B^\mu(\phi):={1\over 8\pi}\int_{\mathbb{R}^3}|\nabla\phi|^2dx+\int_{B_{\mu}}\Psi_{\rho_{\mu}}^*(P\phi-\phi) dx,\quad \phi\in \dot{H}^1(\mathbb{R}^3),
\end{align}
 where $P$ is given in \eqref{ProjectionP}. Then we study the regularity of $A$ and $B$.
\begin{lemma}\label{c2dual functional}
 $A\in C^2(\dot{H}^1(\mathbb{R}^3))$. Consequently, $B\in C^2(\dot{H}^1(\mathbb{R}^3))$,  $B(0)=0, B'(0)=0,$ and
\begin{align*}
\;\;\langle B''(0)\phi,\phi\rangle
%=&{1\over4\pi}\int_{\mathbb{R}^3}|\nabla\phi|^2dx+\int_{B_{\mu}}(\Psi_{\rho_{\mu}}^*)''(0)(P\phi-\phi)^2 dx\\
=&{1\over4\pi}\int_{\mathbb{R}^3}|\nabla\phi|^2dx-\int_{B_{\mu}}{1\over \Phi''(\rho_{\mu})}(P\phi-\phi)^2 dx,\quad \phi\in\dot{H}^1(\mathbb{R}^3).
\end{align*}
\end{lemma}
\begin{proof} For $\phi_{0} \in\dot{H}^1(\mathbb{R}^3)$ and $\psi\in\dot{H}^1(\mathbb{R}^3)$, we have
\begin{align*}
\partial_{\lambda}A(\phi_{0}+\lambda\psi)|_{\lambda=0}&=\int_{B_{\mu}}(\Psi_{\rho_{\mu}}^*)'(P\phi_{0}-\phi_{0})(P\psi-\psi)dx\\
&=\int_{B_{\mu}}\left(-(\Phi')_+^{-1}\left(P\phi_{0}-\phi_{0}+\Phi'(\rho_{\mu})\right)+\rho_{\mu}\right)(P\psi-\psi)dx.
\end{align*}
Since $P'(\rho)>C\rho^{\gamma_1-1}$ for $\rho>N$, we have
$$\Phi'(\rho)=\int_{\mu}^N{P'(s)\over s}ds +\int_N^\rho{P'(s)\over s}ds\geq C(\rho^{\gamma_1-1}-N^{\gamma_1-1}).$$
Thus, we have $(\Phi')_+^{-1}(y)\leq Cy^{1\over \gamma_1-1}$ for $y>N_1$, where $N_1>0$ is a constant large enough. Then  by  the fact that  $\rho_{\mu}\in C^1(B_{\mu})$, we have
\begin{align*}
&|\partial_{\lambda}A(\phi_{0}+\lambda\psi)|_{\lambda=0}|\\
\leq&\int_{B_{\mu}}\left|-(\Phi')_+^{-1}\left(P\phi_{0}-\phi_{0}+\Phi'(\rho_{\mu})\right)+\rho_{\mu}\right||P\psi-\psi|dx\\
\leq&\int_{B_{\mu}\cap\{x|P\phi_{0}-\phi_{0}+\Phi'(\rho_{\mu})\leq N_1\}}|(\Phi')_+^{-1}\left(P\phi_{0}-\phi_{0}+\Phi'(\rho_{\mu})\right)||P\psi-\psi|dx\\
&+\int_{B_{\mu}\cap\{x|P\phi_{0}-\phi_{0}+\Phi'(\rho_{\mu})>N_1\}}|(\Phi')_+^{-1}\left(P\phi_{0}-\phi_{0}+\Phi'(\rho_{\mu})\right)||P\psi-\psi|dx\\
&+C\int_{B_{\mu}}|P\psi-\psi|dx\\
=&II_1+II_2+II_3.
\end{align*}
Since  $(\Phi')^{-1}_+$ is increasing on $[0,\infty)$, by \eqref{P} we have
\begin{align*}
II_1+II_3\leq&|(\Phi')_+^{-1}\left(N_1\right)|\int_{B_{\mu}\cap\{x|P\phi_{0}-\phi_{0}+\Phi'(\rho_{\mu})\leq N_1\}}|P\psi-\psi|dx
+C\int_{B_{\mu}}|P\psi-\psi|dx\\
\leq&C\|P\psi-\psi\|_{L^6(B_{\mu})}\leq C(|P\psi|+\|\psi\|_{L^6(\mathbb{R}^3)})\leq C\|\nabla\psi\|_{L^2(\mathbb{R}^3)}.
\end{align*}
Since $(\Phi')_+^{-1}(y)\leq Cy^{1\over \gamma_1-1}$ for $y>N_1$, we have
\begin{align*}
II_2\leq& C\int_{B_{\mu}\cap\{x|P\phi_{0}-\phi_{0}+\Phi'(\rho_{\mu})>N_1\}}\left(P\phi_{0}-\phi_{0}+\Phi'(\rho_{\mu})\right)^{1\over \gamma_1-1}|P\psi-\psi|dx\\
\leq& \int_{B_{\mu}\cap\{x|P\phi_{0}-\phi_{0}+\Phi'(\rho_{\mu})>N_1\}}C\left|P\phi_{0}-\phi_{0}\right|^{1\over \gamma_1-1}|P\psi-\psi|+C\Phi'(\rho_{\mu})^{1\over \gamma_1-1}|P\psi-\psi|dx\\
\leq&C\left(\int_{B_{\mu}\cap\{x|P\phi_{0}-\phi_{0}+\Phi'(\rho_{\mu})>N_1\}}\left|P\phi_{0}-\phi_{0}\right|^{{6\over 5(\gamma_1-1)}}dx\right)^ {5\over6}\|P\psi-\psi\|_{L^6(B_{\mu})}\\
&+C\|P\psi-\psi\|_{L^1(B_{\mu})}\\
\leq&C\|P\phi_{0}-\phi_{0}\|_{L^{{6\over 5(\gamma_1-1)}}{(B_{\mu})}}^{1\over \gamma_1-1}\|P\psi-\psi\|_{L^6(B_{\mu})}+C\|P\psi-\psi\|_{L^1(B_{\mu})}\\
\leq&C\|P\phi_{0}-\phi_{0}\|_{L^6{(B_{\mu})}}^{1\over \gamma_1-1}\|P\psi-\psi\|_{L^6(B_{\mu})}+C\|P\psi-\psi\|_{L^6(B_{\mu})}\\
\leq&C(|P\phi_{0}|+\|\phi_{0}\|_{L^6{(B_{\mu})}})^{1\over \gamma_1-1}(|P\psi|+\|\psi\|_{L^6(B_{\mu})})+C(|P\psi|+\|\psi\|_{L^6(B_{\mu})})\\
\leq&C\|\nabla\phi_{0}\|_{L^2{(\mathbb{R}^3)}}^{1\over \gamma_1-1}\|\nabla\psi\|_{L^2(\mathbb{R}^3)}+C\|\nabla\psi\|_{L^2(\mathbb{R}^3)},
\end{align*}
where we use the fact that ${6\over 5(\gamma_1-1)}\in({6\over5},6)$.
Thus,
\begin{align*}
|\partial_{\lambda}A(\phi_{0}+\lambda\psi)|_{\lambda=0}|\leq II_1+II_2+II_3\leq C\|\nabla\phi_{0}\|_{L^2{(\mathbb{R}^3)}}^{1\over \gamma_1-1}\|\nabla\psi\|_{L^2(\mathbb{R}^3)}+C\|\nabla\psi\|_{L^2(\mathbb{R}^3)}.
\end{align*}
This implies that $A$ is G$\hat{\text{a}}$teaux differentiable at $\phi_{0}\in \dot{H}^1(\mathbb{R}^3)$.
To show that $A\in C^1(\dot{H}^1(\mathbb{R}^3))$, we choose $\{\phi_n\}\subset\dot{H}^1(\mathbb{R}^3)$  such that $\phi_n\to\phi_{0}$ in $\dot{H}^1(\mathbb{R}^3)$, and prove that
 \begin{align}\label{derivative continuous}
\partial_{\lambda}A(\phi_n+\lambda\psi)|_{\lambda=0}\to\partial_{\lambda}A(\phi_{0}+\lambda\psi)|_{\lambda=0}\quad{\text{as}}\quad n\to\infty.
\end{align}
In fact, by \eqref{z(y)} we have
\begin{align*}
&|\partial_{\lambda}A(\phi_n+\lambda\psi)|_{\lambda=0}-\partial_{\lambda}A(\phi_{0}+\lambda\psi)|_{\lambda=0}|\\
\leq&\int_{B_{\mu}}\left|(\Psi^*_{\rho_{\mu}})'(P\phi_n-\phi_n)-(\Psi^*_{\rho_{\mu}})'(P\phi_{0}-\phi_{0})\right||P\psi-\psi|dx\\
=&\int_{B_{\mu}}\left|(\Psi^*_{\rho_{\mu}})''(\varphi_n)(P(\phi_n-\phi_{0})-(\phi_n-\phi_{0}))\right||P\psi-\psi|dx\\
%=&\int_{B_{\mu}}{1\over\Phi''\left(z(\varphi_n)+\rho_{\mu}\right)}|P(\phi_n-\phi_{0})-(\phi_n-\phi_{0})||P\psi-\psi|dx\\
=&\int_{B_{\mu}}{1\over\Phi''\left((\Phi')_+^{-1}\left(\varphi_n+\Phi'(\rho_{\mu})\right)\right)}|P(\phi_n-\phi_{0})-(\phi_n-\phi_{0})||P\psi-\psi|dx,
\end{align*}
where there exists $s_n(x)\in(0,1)$ such that  
$$\varphi_n(x)=s_n(x)(P\phi_n-\phi_n(x))+(1-s_n(x))(P\phi_{0}-\phi_{0}(x))$$
 for fixed $x\in B_{\mu}$ and $n\in\mathbb{Z}^+$.
Since $\Phi''(s)={P'(s)\over s}>Cs^{\gamma_1-2}$ for $s>N$, we have $(\Phi')_+^{-1}(y)\to\infty$ as $y\to\infty$, and ${1\over \Phi''(s)}<Cs^{2-\gamma_1}$ for $s>N$. Since $(\Phi')_+^{-1}$ is increasing on $[0,\infty)$, there exists $N_2>N_1$ such that $(\Phi')_+^{-1}(y)>N$ for $y>N_2$. Noting that  $(\Phi')_+^{-1}(y)\leq Cy^{1\over \gamma_1-1}$ for $y>N_1$, we have for $y>N_2$,
\begin{align}\label{2derivativePhi-1derivativePhi-1}
{1\over\Phi''((\Phi')_+^{-1}(y))}\leq C\left((\Phi')_+^{-1}(y)\right)^{2-\gamma_1}\leq Cy^{{2-\gamma_1\over \gamma_1-1}}.
\end{align}
Thus,
\begin{align*}
&|\partial_{\lambda}A(\phi_n+\lambda\psi)|_{\lambda=0}-\partial_{\lambda}A(\phi_{0}+\lambda\psi)|_{\lambda=0}|\\
\leq&\int_{B_{\mu}\cap\{x|\varphi_n+\Phi'(\rho_{\mu})\leq N_2\}}{1\over\Phi''\left((\Phi')_+^{-1}\left(\varphi_n+\Phi'(\rho_{\mu})\right)\right)}|P(\phi_n-\phi_{0})-(\phi_n-\phi_{0})||P\psi-\psi|dx\\
&+\int_{B_{\mu}\cap\{x|\varphi_n+\Phi'(\rho_{\mu})> N_2\}}{1\over\Phi''\left((\Phi')_+^{-1}\left(\varphi_n+\Phi'(\rho_{\mu})\right)\right)}|P(\phi_n-\phi_{0})-(\phi_n-\phi_{0})||P\psi-\psi|dx\\
\leq &\sup_{y\leq N_2}{1\over\Phi''\left((\Phi')_+^{-1}\left(y\right)\right)}\int_{B_{\mu}\cap\{x|\varphi_n+\Phi'(\rho_{\mu})\leq N_2\}}|P(\phi_n-\phi_{0})-(\phi_n-\phi_{0})||P\psi-\psi|dx\\
&+C\int_{B_{\mu}\cap\{x|\varphi_n+\Phi'(\rho_{\mu})> N_2\}}\left(\varphi_n+\Phi'(\rho_{\mu})\right)^{{2-\gamma_1\over \gamma_1-1}}|P(\phi_n-\phi_{0})-(\phi_n-\phi_{0})||P\psi-\psi|dx\\
\leq& C\|P(\phi_n-\phi_{0})-(\phi_n-\phi_{0})\|_{L^2(B_{\mu})}\|P\psi-\psi\|_{L^2(B_{\mu})}\\
&+C\left(\int_{B_{\mu}}\left|\varphi_n+\Phi'(\rho_{\mu})\right|^{{3(2-\gamma_1)\over2( \gamma_1-1)}}dx\right)^{2\over3}
\|P(\phi_n-\phi_{0})-(\phi_n-\phi_{0})\|_{L^6(B_{\mu})}\|P\psi-\psi\|_{L^6(B_{\mu})}\\
\leq& C(|P(\phi_n-\phi_{0})|+\|\phi_n-\phi_{0}\|_{L^2(B_{\mu})})(|P\psi|+\|\psi\|_{L^2(B_{\mu})})\\
&+C\|\varphi_n+\Phi'(\rho_{\mu})\|_{L^{{3(2-\gamma_1)\over2( \gamma_1-1)}}(B_{\mu})}^{{2-\gamma_1\over \gamma_1-1}}
(|P(\phi_n-\phi_{0})|+\|\phi_n-\phi_{0}\|_{L^6(B_{\mu})})(|P\psi|+\|\psi\|_{L^6(B_{\mu})})\\
\leq& C\|\nabla(\phi_n-\phi_{0})\|_{L^2(\mathbb{R}^3)}\|\nabla\psi\|_{L^2(\mathbb{R}^3)}
+C\|\varphi_n+\Phi'(\rho_{\mu})\|_{L^{6}(B_{\mu})}^{{2-\gamma_1\over \gamma_1-1}}
\|\nabla(\phi_n-\phi_{0})\|_{L^2(\mathbb{R}^3)}\|\nabla\psi\|_{L^2(\mathbb{R}^3)},
\end{align*}
where we use the fact that ${3(2-\gamma_1)\over2( \gamma_1-1)}\in(0,6)$.
Since $\phi_n\to\phi_{0}$ in $\dot{H}^1(\mathbb{R}^3)$, we have $\|\nabla\phi_n\|_{L^2(\mathbb{R}^3)}\leq C_{\|\nabla\phi_{0}\|_{L^2(\mathbb{R}^3)}}$ for $n$ large enough. Then
\begin{align*}
\|\varphi_n+\Phi'(\rho_{\mu})\|_{L^{6}(B_{\mu})}=&\|s_n(P\phi_n-\phi_n)+(1-s_n)(P\phi_{0}-\phi_{0})+\Phi'(\rho_{\mu})\|_{L^{6}(B_{\mu})}\\
\leq&|P\phi_n|+\|\phi_n(x)\|_{L^{6}(B_{\mu})}+|P\phi_{0}|+\|\phi_{0}(x)\|_{L^{6}(B_{\mu})}+C\\
\leq&C\|\nabla\phi_n\|_{L^2(\mathbb{R}^3)}+C\|\nabla\phi_{0}\|_{L^2(\mathbb{R}^3)}+C\leq C_{\|\nabla\phi_{0}\|_{L^2(\mathbb{R}^3)}}.
\end{align*}
Thus,
\begin{align}\nonumber
&|\partial_{\lambda}A(\phi_n+\lambda\psi)|_{\lambda=0}-\partial_{\lambda}A(\phi_{0}+\lambda\psi)|_{\lambda=0}|\\
\leq& C_{\|\nabla\phi_{0}\|_{L^2(\mathbb{R}^3)}}\|\nabla(\phi_n-\phi_{0})\|_{L^2(\mathbb{R}^3)}\|\nabla\psi\|_{L^2(\mathbb{R}^3)}\to 0,\text{ as } n\to \infty.\label{A-C1}
\end{align}
This proves \eqref{derivative continuous} and thus, $A\in C^1(\dot{H}^1(\mathbb{R}^3))$.

Next, we show that the  2-th order G$\hat{\text{a}}$teaux derivative of $A$ exists at $\phi_{0} \in\dot{H}^1(\mathbb{R}^3)$.
For $\varphi \in\dot{H}^1(\mathbb{R}^3)$ and $\psi\in\dot{H}^1(\mathbb{R}^3)$, by \eqref{z(y)} we have
\begin{align*}
&\partial_{\tau}\partial_{\lambda}A(\phi_{0}+\lambda\psi+\tau \varphi )|_{\lambda=\tau=0}\\
=&\int_{B_{\mu}}(\Psi_{\rho_{\mu}}^*)''(P\phi_{0}-\phi_{0})(P\psi-\psi)(P\varphi-\varphi)dx\\
=&\int_{B_{\mu}}-{1\over\Phi''\left((\Phi')_+^{-1}\left(P\phi_{0}-\phi_{0}+\Phi'(\rho_{\mu})\right)\right)}(P\psi-\psi)(P\varphi-\varphi)dx.
\end{align*}
By \eqref{2derivativePhi-1derivativePhi-1} and the fact that ${3(2-\gamma_1)\over2( \gamma_1-1)}\in(0,6)$, we have
\begin{align*}
&|\partial_{\tau}\partial_{\lambda}A(\phi_{0}+\lambda\psi+\tau \varphi )|_{\lambda=\tau=0}|\\
\leq&\int_{B_{\mu}\cap\{x|P\phi_{0}-\phi_{0}+\Phi'(\rho_{\mu})\leq N_2\}}{1\over\Phi''\left((\Phi')_+^{-1}\left(P\phi_{0}-\phi_{0}+\Phi'(\rho_{\mu})\right)\right)}|P\psi-\psi||P\varphi-\varphi|dx\\
&+\int_{B_{\mu}\cap\{x|P\phi_{0}-\phi_{0}+\Phi'(\rho_{\mu})>N_2\}}{1\over\Phi''\left((\Phi')_+^{-1}\left(P\phi_{0}-\phi_{0}+\Phi'(\rho_{\mu})\right)\right)}|P\psi-\psi||P\varphi-\varphi|dx\\
\leq&\sup_{y\leq N_2}{1\over\Phi''\left((\Phi')_+^{-1}\left(y\right)\right)}\int_{B_{\mu}\cap\{x|P\phi_{0}-\phi_{0}+\Phi'(\rho_{\mu})\leq N_2\}}|P\psi-\psi||P\varphi-\varphi|dx\\
&+C\int_{B_{\mu}\cap\{x|P\phi_{0}-\phi_{0}+\Phi'(\rho_{\mu})> N_2\}}\left|P\phi_{0}-\phi_{0}+\Phi'(\rho_{\mu})\right|^{{2-\gamma_1\over \gamma_1-1}}|P\psi-\psi||P\varphi-\varphi|dx\\
\leq &C\|P\psi-\psi\|_{L^2(B_{\mu})}\|P\varphi-\varphi\|_{L^2(B_{\mu})}\\
&+\|P\phi_{0}-\phi_{0}+\Phi'(\rho_{\mu})\|_{L^{{3(2-\gamma_1)\over2( \gamma_1-1)}}(B_{\mu})}^{{2-\gamma_1\over \gamma_1-1}}\|P\psi-\psi\|_{L^6(B_{\mu})}\|P\varphi-\varphi\|_{L^6(B_{\mu})}\\
\leq &C\|\nabla\psi\|_{L^2(\mathbb{R}^3)}\|\nabla\varphi\|_{L^2(\mathbb{R}^3)}
+(\|P\phi_{0}-\phi_{0}\|_{L^{6}(B_{\mu})}+C)^{{2-\gamma_1\over \gamma_1-1}}\|\nabla\psi\|_{L^2(\mathbb{R}^3)}\|\nabla\varphi\|_{L^2(\mathbb{R}^3)}\\
\leq &(C+(C\|\nabla\phi_{0}\|_{L^{2}(\mathbb{R}^3)}+C)^{{2-\gamma_1\over \gamma_1-1}})\|\nabla\psi\|_{L^2(\mathbb{R}^3)}\|\nabla\varphi\|_{L^2(\mathbb{R}^3)}\\
\leq&(C+C_{\|\nabla\phi_{0}\|_{L^{2}(\mathbb{R}^3)}})\|\nabla\psi\|_{L^2(\mathbb{R}^3)}\|\nabla\varphi\|_{L^2(\mathbb{R}^3)}.
\end{align*}
Thus,  $A$ is 2-th G$\hat{\text{a}}$teaux differentiable at $\phi_{0}\in \dot{H}^1(\mathbb{R}^3)$.

We will prove that
 \begin{align}\label{derivative continuous2}
\partial_{\tau}\partial_{\lambda}A(\phi_n+\lambda\psi+\tau \varphi )|_{\lambda=\tau=0}\to\partial_{\tau}\partial_{\lambda}A(\phi_{0}+\lambda\psi+\tau \varphi )|_{\lambda=\tau=0}
\end{align}
for  $\{\phi_n\}\subset\dot{H}^1(\mathbb{R}^3)$  such that $\phi_n\to\phi_{0}$ in $\dot{H}^1(\mathbb{R}^3)$, which implies that
$A\in C^2(\dot{H}^1(\mathbb{R}^3))$.
Since the regularity of  $\Phi$ is  $C^2(0,\infty)$, we can not expect the Lipschitz estimates like  \eqref{A-C1}, and we have to prove \eqref{derivative continuous2} by definition. Fix any $0<\varepsilon<1$. Since  $\phi_{0}\in L^6(B_{\mu})$, there exists $\delta\in(0,\varepsilon)$ such that
$\|\phi_{0}\|_{L^6(B)}<\varepsilon$ for $B\subset B_{\mu}$ with $|B|<\delta$.
 By Lusin's Theorem, there exists a  subset $B_{0,0}\subset B_{\mu}$
such that $B_{\mu}\setminus B_{0,0}$ is closed, $|B_{0,0}| < {\delta\over2}$ and $\phi_0$ is continuous on $B_{\mu}\setminus B_{0,0}$.
Since   ${1\over\Phi''\left((\Phi')_+^{-1}\left(\cdot\right)\right)}\in C^0(\mathbb{R})$, we can choose $\delta\in(0,\varepsilon)$ small enough such that
$\left|{1\over\Phi''\left((\Phi')_+^{-1}\left(y_2\right)\right)}-{1\over\Phi''\left((\Phi')_+^{-1}\left(y_1\right)\right)}\right|<\varepsilon$ for  $|y_1-y_2|<\delta$ and $$y_1,y_2\in[\min\limits_{x\in B_{\mu}\setminus B_{0,0}}\left(P\phi_{0}-\phi_{0}(x)+\Phi'(\rho_{\mu}(x))\right)-1,\max\limits_{x\in B_{\mu}\setminus B_{0,0}}\left(P\phi_{0}-\phi_{0}(x)+\Phi'(\rho_{\mu}(x))\right)+1].$$
Since $\phi_n\to\phi_{0}$ in $\dot{H}^1(\mathbb{R}^3)$, there exists $\tilde N_{\mu}>0$ large enough such that $$\|\phi_n-\phi_{0}\|_{L^6(B_{\mu})}\leq\|\phi_n-\phi_{0}\|_{L^6(\mathbb{R}^3)}\leq C\|\nabla(\phi_n-\phi_{0})\|_{L^2(\mathbb{R}^3)}\leq \left({\delta\over2}\right)^{7\over6}$$
 and
  $|P\phi_n-P\phi_{0}|\leq C\|\nabla(\phi_n-\phi_{0})\|_{L^2(\mathbb{R}^3)}\leq {1\over2}\delta$ for $n>\tilde N_{\mu}$.
Define
\begin{align*}
B_{0,1}^n=\left\{x\in B_{\mu}\setminus B_{0,0}||\phi_n(x)-\phi_{0}(x)|<{1\over2}\delta\right\},
B_{0,2}^n=\left\{x\in B_{\mu}\setminus B_{0,0}||\phi_n(x)-\phi_{0}(x)|\geq{1\over2}\delta\right\}
\end{align*}
for fixed $n>\tilde N_{\mu}$. Then
\begin{align*}
{1\over2}|B_{0,2}^n|^{1\over6}\delta\leq \|\phi_n-\phi_{0}\|_{L^6(B_{0,2}^n)}\leq \|\phi_n-\phi_{0}\|_{L^6(B_{\mu})}\leq \left({\delta\over2}\right)^{7\over6}\Longrightarrow|B_{0,2}^n|\leq{\delta\over2}<{\varepsilon\over2},
\end{align*}
for $n>\tilde N_{\mu}$. Let
\begin{align*}
\tilde \phi_n=P\phi_n-\phi_n+\Phi'(\rho_{\mu})\quad\text{and}\quad \tilde \phi_{0}=P\phi_{0}-\phi_{0}+\Phi'(\rho_{\mu}).
\end{align*}
Then for $n>\tilde N_{\mu}$ and $x\in B_{0,1}^n$, we have
 \begin{align*}
|\tilde \phi_n(x)-\tilde \phi_{0}(x)|\leq |P\phi_n-P\phi_{0}|+|\phi_n(x)-\phi_{0}(x)|\leq{1\over2}\delta+{1\over2}\delta=\delta,
\end{align*}
and thus, \begin{align*}
\left|{1\over\Phi''\left((\Phi')_+^{-1}\left(\tilde \phi_n(x)\right)\right)}-{1\over\Phi''\left((\Phi')_+^{-1}\left(\tilde \phi_{0}(x)\right)\right)}\right|<\varepsilon.
\end{align*}
Then $|B_{0,0}\cup B_{0,2}^n|\leq {\delta\over2}+{\delta\over2}=\delta<\varepsilon$, and
 \begin{align*}
&|\partial_{\tau}\partial_{\lambda}A(\phi_n+\lambda\psi+\tau \varphi )|_{\lambda=\tau=0}-\partial_{\tau}\partial_{\lambda}A(\phi_{0}+\lambda\psi+\tau \varphi )|_{\lambda=\tau=0}|\\
\leq& \int_{B_{0,1}^n}\left|{1\over\Phi''\left((\Phi')_+^{-1}\left(\tilde \phi_n\right)\right)}
-{1\over\Phi''\left((\Phi')_+^{-1}\left(\tilde \phi_{0}\right)\right)}\right||P\psi-\psi||P\varphi-\varphi|dx\\
&+ \int_{B_{0,0}\cup B_{0,2}^n}\left|{1\over\Phi''\left((\Phi')_+^{-1}\left(\tilde \phi_n\right)\right)}
-{1\over\Phi''\left((\Phi')_+^{-1}\left(\tilde \phi_{0}\right)\right)}\right||P\psi-\psi||P\varphi-\varphi|dx\\
\leq& \varepsilon\int_{B_{0,1}^n}|P\psi-\psi||P\varphi-\varphi|dx
+ \int_{B_{0,0}\cup B_{0,2}^n}\left|{1\over\Phi''\left((\Phi')_+^{-1}\left(\tilde \phi_n\right)\right)}\right||P\psi-\psi||P\varphi-\varphi|dx\\
&+\int_{B_{0,0}\cup B_{0,2}^n}\left|{1\over\Phi''\left((\Phi')_+^{-1}\left(\tilde \phi_{0}\right)\right)}\right||P\psi-\psi||P\varphi-\varphi|dx\\
=&III_1+III_2+III_3.
\end{align*}
For $III_1$, we have
\begin{align*}
III_1\leq \varepsilon\|\nabla\psi\|_{L^2(\mathbb{R}^3)}\|\nabla\varphi\|_{L^2(\mathbb{R}^3)}.
\end{align*}
To estimate $III_2$ and $III_3$, by the fact that  $|P\phi_{0}|\leq C_{\|\nabla\phi_{0}\|_{L^2(\mathbb{R}^3)}}<\infty$ we have
 %and $N_{\mu}>0$ can be chosen large enough such that $\|\nabla(\phi_n-\phi_{0})\|_{L^{2}(\mathbb{R}^3)}<\varepsilon$ for $n>N_{\mu}$.
\begin{align*}
 \|\tilde \phi_{0}\|_{L^{6}(B_{0,0}\cup B_{0,2}^n)}\leq& \|\phi_{0}\|_{L^{6}(B_{0,0}\cup B_{0,2}^n)}+C_{\|\nabla\phi_{0}\|_{L^2(\mathbb{R}^3)}}|B_{0,0}\cup B_{0,2}^n|^{1\over6}\leq \varepsilon+C_{\|\nabla\phi_{0}\|_{L^2(\mathbb{R}^3)}}\varepsilon^{1\over6},\\
 \|\tilde \phi_n\|_{L^{6}(B_{0,0}\cup B_{0,2}^n)}\leq&  \|\tilde \phi_n-\tilde \phi_{0}\|_{L^{6}(B_{0,0}\cup B_{0,2}^n)}+ \|\tilde \phi_{0}\|_{L^{6}(B_{0,0}\cup B_{0,2}^n)}\\
 \leq&
 \|P (\phi_n-\phi_{0})+(\phi_n-\phi_{0})\|_{L^{6}(B_{0,0}\cup B_{0,2}^n)}+\varepsilon+C_{\|\nabla\phi_{0}\|_{L^2(\mathbb{R}^3)}}\varepsilon^{1\over6}\\
 \leq& C\varepsilon+\varepsilon+C_{\|\nabla\phi_{0}\|_{L^2(\mathbb{R}^3)}}\varepsilon^{1\over6}
\end{align*}
for $n>\tilde N_{\mu}$.
For $III_2$, by \eqref{2derivativePhi-1derivativePhi-1} we have
\begin{align*}
III_2\leq& \int_{(B_{0,0}\cup B_{0,2}^n)\cap\{x|\tilde \phi_n(x)\leq N_2\}}{1\over\Phi''\left((\Phi')_+^{-1}\left(\tilde \phi_n\right)\right)}|P\psi-\psi||P\varphi-\varphi|dx\\
&+\int_{(B_{0,0}\cup B_{0,2}^n)\cap\{x|\tilde \phi_n(x)> N_2\}}{1\over\Phi''\left((\Phi')_+^{-1}\left(\tilde \phi_n\right)\right)}|P\psi-\psi||P\varphi-\varphi|dx\\
\leq& C\int_{(B_{0,0}\cup B_{0,2}^n)\cap\{x|\tilde \phi_n(x)\leq N_2\}}|P\psi-\psi||P\varphi-\varphi|dx+\\
&\|\tilde \phi_n\|_{L^{{3(2-\gamma_1)\over2( \gamma_1-1)}}(B_{0,0}\cup B_{0,2}^n)}^{{2-\gamma_1\over \gamma_1-1}}\|P\psi-\psi\|_{L^6(B_{0,0}\cup B_{0,2}^n)}\|P\varphi-\varphi\|_{L^6(B_{0,0}\cup B_{0,2}^n)}\\
\leq &C|B_{0,0}\cup B_{0,2}^n|^{2\over3}\|P\psi-\psi\|_{L^6(B_{0,0}\cup B_{0,2}^n)}\|P\varphi-\varphi\|_{L^6(B_{0,0}\cup B_{0,2}^n)}\\
&+\|\tilde \phi_n\|_{L^{6}(B_{0,0}\cup B_{0,2}^n)}^{{2-\gamma_1\over \gamma_1-1}}\|\nabla\psi\|_{L^2(\mathbb{R}^3)}\|\nabla\varphi\|_{L^2(\mathbb{R}^3)}\\
\leq& \left(C\varepsilon^{2\over3}+(C\varepsilon+\varepsilon+C_{\|\nabla\phi_{0}\|_{L^2(\mathbb{R}^3)}}\varepsilon^{1\over6})^{^{{2-\gamma_1\over \gamma_1-1}}}\right)\|\nabla\psi\|_{L^2(\mathbb{R}^3)}\|\nabla\varphi\|_{L^2(\mathbb{R}^3)}.
\end{align*}
Similar to  $III_2$,    we have
\begin{align*}
III_3
\leq& \left(C\varepsilon^{2\over3}+(\varepsilon+C_{\|\nabla\phi_{0}\|_{L^2(\mathbb{R}^3)}}\varepsilon^{1\over6})^{^{{2-\gamma_1\over \gamma_1-1}}}\right)\|\nabla\psi\|_{L^2(\mathbb{R}^3)}\|\nabla\varphi\|_{L^2(\mathbb{R}^3)}.
\end{align*}
Then \eqref{derivative continuous2} follows from the estimates for $III_1$, $III_2$ and $III_3$.
\end{proof}

\section{Kernel and negative directions  of $B''(0)$}
Define
\begin{align}\label{def-tilde-L-mu}
\tilde L_\mu=-\Delta-{4\pi\over \Phi''(\rho_{\mu})}(I-P): \dot{H}^1(\mathbb{R}^3)\to\dot{H}^{-1}(\mathbb{R}^3).
\end{align}

For $\phi\in\dot{H}^1(\mathbb{R}^3)$, we have \begin{align*}
\int_{B_{\mu}}{1\over \Phi''(\rho_{\mu})}(\phi-P\phi)\phi dx=\int_{B_{\mu}}{1\over \Phi''(\rho_{\mu})}(\phi-P\phi)^2dx,
\end{align*} and thus,
\begin{align*}
\langle B''(0)\phi,\phi\rangle={1\over 4\pi}\langle\tilde L_\mu\phi,\phi\rangle.
\end{align*}

Consider the quadratic form
\begin{align}\label{density quadratic form}
\;\;\langle L_{\mu}\rho,\rho\rangle=&\int_{B_{\mu}}\Phi''(\rho_{\mu})\rho^2dx-{1\over 4\pi}\int_{\mathbb{R}^3}|\nabla V|^2dx,\quad\rho\in Y_{\mu}:=L^2_{\Phi''(\rho_{\mu})}(B_{\mu})
\end{align}
with $\Delta V=4\pi\rho$. By (3.21) in \cite{LZ2022}, we have  $V=4\pi\Delta^{-1}\rho\in \dot{H}^1(\mathbb{R}^3)$ for $\rho\in Y_{\mu}$, and thus,  the quadratic form \eqref{density quadratic form} is well-defined. The operator associated to \eqref{density quadratic form} is
\begin{align*}
L_{\mu}=\Phi''(\rho_{\mu})-{4\pi}(-\Delta)^{-1} :\;\; Y_{\mu}\to Y_{\mu}^*.
\end{align*}
We then study the relations of the non-positive   directions and kernel between $\tilde L_\mu$ and  $L_{\mu}|_{Z_\mu}$, where
\begin{align}\label{def-Z-mu}
Z_\mu=\{\rho\in Y_{\mu}|\int_{B_{\mu}}\rho dx=0\}.
\end{align}
\begin{lemma}\label{negative-zero-direction}
\begin{align*}
 n^{\leq 0}(\tilde L_\mu)   =n^{\leq0}(L_{\mu}|_{Z_\mu})\quad\text{and} \quad
  \dim(\ker(\tilde L_\mu))   =\dim(\ker(L_{\mu}|_{Z_\mu})),
  \end{align*}
  where $n^{\leq0}(\tilde L_\mu)$ is the  non-positive dimension
of  $\langle\tilde L_\mu\cdot,\cdot\rangle$ and $n^{\leq0}(L_{\mu}|_{Z_\mu})$ is  the  non-positive dimension
of  $\langle L_\mu\cdot,\cdot\rangle$ restricted to $Z_\mu$.
\end{lemma}
\begin{proof}
For  $\rho\in Z_\mu$ with $\Delta V=4\pi\rho$, we have
\begin{align*}
\langle L_{\mu}\rho,\rho\rangle=&{1\over 4\pi}\int_{\mathbb{R}^3}|\nabla V|^2dx+\int_{B_{\mu}}\Phi''(\rho_{\mu})\rho^2dx+\int_{B_{\mu}}2V\rho dx\\
=&{1\over 4\pi}\int_{\mathbb{R}^3}|\nabla V|^2dx+\int_{B_{\mu}}\Phi''(\rho_{\mu})\rho^2dx+\int_{B_{\mu}}2(V-PV)\rho dx\\
\geq&{1\over 4\pi}\int_{\mathbb{R}^3}|\nabla V|^2dx-\int_{B_{\mu}}{1\over\Phi''(\rho_{\mu})}(V-PV)^2 dx\\
=&{1\over 4\pi}\int_{\mathbb{R}^3}|\nabla V|^2dx-\int_{B_{\mu}}{1\over\Phi''(\rho_{\mu})}(V-PV)V dx={1\over 4\pi}\langle\tilde L_\mu V,V\rangle.
\end{align*}
Thus, $n^{\leq 0}(\tilde L_\mu)  \geq n^{\leq0}(L_{\mu}|_{Z_\mu})$. For $\phi\in \dot{H}^1(\mathbb{R}^3)$, let $\rho_\phi={1\over \Phi''(\rho_{\mu})}(\phi-P\phi)$.
%where $\left({1\over \Phi''(\rho_{\mu}(r))}\right)_+={1\over \Phi''(\rho_{\mu}(r))}$ for $0\leq r\leq R$ and $\left({1\over \Phi''(\rho_{\mu})}\right)_+=0$ for $r>R$.
Then
$\int_{B_{\mu}}{ \Phi''(\rho_{\mu})}\rho_\phi^2 dx=\int_{B_{\mu}}{1\over \Phi''(\rho_{\mu})}(P\phi-\phi)^2 dx\leq C\|\nabla\phi\|_{L^2(\mathbb{R}^3)}^2$ and $\int_{B_{\mu}}\rho_\phi dx=0$. Thus, $\rho_\phi\in Z_\mu$ and $V_\phi:=4\pi\Delta^{-1}\rho_\phi\in\dot{H}^1(\mathbb{R}^3)$. Then
\begin{align*}
\notag\langle\tilde L_\mu\phi,\phi\rangle=&\int_{\mathbb{R}^3}|\nabla\phi|^2dx-\int_{B_{\mu}}{4\pi\over \Phi''(\rho_{\mu})}(\phi-P\phi)^2 dx\\
\notag=&\int_{\mathbb{R}^3}|\nabla\phi|^2dx-4\pi\int_{B_{\mu}}{ \Phi''(\rho_{\mu})}\rho_\phi^2 dx\\
\notag=&4\pi\int_{B_{\mu}}{ \Phi''(\rho_{\mu})}\rho_\phi^2 dx+\int_{\mathbb{R}^3}|\nabla\phi|^2dx-8\pi\int_{B_{\mu}}(\phi-P\phi)\rho_\phi dx\\
\notag=&4\pi\int_{B_{\mu}}{ \Phi''(\rho_{\mu})}\rho_\phi^2 dx+\int_{\mathbb{R}^3}|\nabla\phi|^2dx+2\int_{\mathbb{R}^3}\nabla\phi\cdot \nabla V_\phi dx\\
\geq&4\pi\int_{B_{\mu}}{ \Phi''(\rho_{\mu})}\rho_\phi^2 dx-\int_{\mathbb{R}^3}|\nabla V_\phi |^2dx=4\pi\langle L_{\mu}\rho_\phi,\rho_\phi\rangle.
\end{align*}
Thus, $n^{\leq 0}(\tilde L_\mu)  \leq n^{\leq0}(L_{\mu}|_{Z_\mu})$. This proves $n^{\leq 0}(\tilde L_\mu)  = n^{\leq0}(L_{\mu}|_{Z_\mu})$.

For $\rho\in \ker(L_{\mu}|_{Z_\mu})$, let $ V=4\pi\Delta^{-1}\rho\in\dot{H}^1(\mathbb{R}^3)$. Then
\begin{align*}
\langle L_{\mu}|_{Z_\mu}\rho,\tilde\rho\rangle=\int_{B_{\mu}}\Phi''(\rho_{\mu})\rho\tilde \rho dx+\int_{\mathbb{R}^3}(V-PV)\tilde \rho dx=0
\end{align*}
for $\tilde \rho \in Z_\mu$. For any $\hat V\in\dot{H}^1(\mathbb{R}^3)$, let $\hat\rho={1\over \Phi''(\rho_{\mu})}(\hat V-P\hat V)\in Z_\mu$. Since
$\int_{B_{\mu}}\rho dx=0$, by the definition of the projection $P$ we have
\begin{align*}
\langle\tilde L_\mu V,\hat V\rangle =& -\int_{B_{\mu}}4\pi\rho(\hat V-P\hat V)dx-\int_{B_{\mu}}{4\pi\over \Phi''(\rho_{\mu})}(V-PV)(\hat V-P\hat V)dx\\
=&-4\pi\left(\int_{B_{\mu}}\Phi''(\rho_{\mu})\rho\hat\rho dx+\int_{B_{\mu}}(V-PV)\hat\rho dx\right)=\langle L_{\mu}|_{Z_\mu}\rho,\hat \rho\rangle=0,
\end{align*}
which gives $\tilde L_\mu V=0\in \dot{H}^{-1}(\mathbb{R}^3)$. Thus, $\dim(\ker(\tilde L_\mu))  \geq\dim(\ker(L_{\mu}|_{Z_\mu})).$

For $\phi\in\ker(\tilde L_\mu)$, let $\rho_\phi={1\over \Phi''(\rho_{\mu})}(\phi-P\phi)\in Z_\mu$. Then
\begin{align*}
\langle \tilde L_\mu\phi,\tilde\phi\rangle=\int_{\mathbb{R}^3}\nabla\phi \cdot\nabla \tilde\phi dx-\int_{B_{\mu}}{4\pi\over \Phi''(\rho_{\mu})}(\phi-P\phi)\tilde \phi dx=0
\end{align*}
for $\tilde \phi\in \dot{H}^{1}(\mathbb{R}^3)$. For $\tilde \rho\in Z_\mu$, let $\tilde V=4\pi\Delta^{-1}\tilde \rho\in \dot{H}^1(\mathbb{R}^3)$. Then
\begin{align*}
\langle L_{\mu}|_{Z_\mu}\rho_\phi,\tilde\rho\rangle =& \int_{B_{\mu}}(\phi-P\phi)\tilde \rho dx+\int_{B_{\mu}}\rho_\phi\tilde Vdx
=\int_{B_{\mu}}\phi\tilde \rho dx+\int_{B_{\mu}}{1\over \Phi''(\rho_{\mu})}(\phi-P\phi)\tilde V dx\\
=&\int_{B_{\mu}}\phi\tilde \rho dx+{1\over 4\pi}\int_{\mathbb{R}^3}\nabla\phi\cdot\nabla\tilde V dx=0,
\end{align*}
which implies $L_{\mu}|_{Z_\mu}\rho_\phi=0\in Z_\mu^*$. Thus, $\dim(\ker(\tilde L_\mu))  \leq\dim(\ker(L_{\mu}|_{Z_\mu})).$ This proves $\dim(\ker(\tilde L_\mu))  =\dim(\ker(L_{\mu}|_{Z_\mu})).$
\end{proof}

Then we study the kernel of $\tilde L_\mu$.
\begin{lemma}\label{kerLmu}
If $M' (\mu)\neq0$, then
\begin{align*}
 \ker\left(\tilde L_\mu\right)=\{\partial_{x^i}V_{\mu}, i=1,2,3\}.
\end{align*}
\end{lemma}

\begin{proof}
Since $\Delta V_\mu=4\pi\rho_\mu=4\pi(\Phi')_+^{-1}(V_\mu(R_{\mu})-V_\mu)$, we have $\Delta (\partial_{x^i}V_\mu)=-{4\pi\over \Phi''(\rho_\mu)}\partial_{x^i}V_\mu$.
Moreover,
 \begin{align*}
 P(\partial_{x^i}V_\mu)=\frac{\int_{B_{\mu}}\frac{\partial_{x^i}V_\mu}{\Phi''(\rho_{\mu})}dx}{\int_{B_{\mu}}\frac{1}{\Phi''(\rho_{\mu})}dx}
 =\frac{-\partial_{x^i}\int_{B_{\mu}}\rho_{\mu}dx}{\int_{B_{\mu}}\frac{1}{\Phi''(\rho_{\mu})}dx}=0.
 \end{align*}
Thus,  $\{\partial_{x^i}V_{\mu}, i=1,2,3\}\subset\ker\left(\tilde L_\mu\right).$
By Lemma \ref{negative-zero-direction}, it suffices to prove that $\dim(\ker(L_{\mu}|_{Z_{\mu}}))=3$.
For $\rho=\rho_{\text{r}}+\rho_{\text{nr}}\in\ker(L_{\mu}|_{Z_\mu})$, we have $L_{\mu}\rho=L_{\mu}\rho_{\text{r}}+L_{\mu} \rho_{\text{nr}}=c$ for some constant $c$, where $\rho_{\text{r}}\in Z_{\mu,\text{r}}:=\{\rho\in Z_\mu|\rho(x)=\rho(|x|)\}$ and $\rho_{\text{nr}}\in Z_{\mu,\text{nr}}:=\{\rho\in Z_\mu|\int_{\mathbb{S}^2}\rho(r\theta)dS_\theta=0\}$ are  the radial  and non-radial parts of $\rho$, respectively. Since $L_{\mu}\rho_{\text{nr}}=c-L_{\mu} \rho_{\text{r}}\in Z_{\mu,\text{r}}\cap  Z_{\mu,\text{nr}}$, we have $L_{\mu}\rho_{\text{nr}}=0$ and $L_{\mu} \rho_{\text{r}}=c$.
%Thus, we divide the discussion into two cases.

Let $\psi_{\text{nr}}=(-\Delta)^{-1}\rho_{\text{nr}}\in \dot{H}^1(\mathbb{R}^3)$. Since $L_{\mu}\rho_{\text{nr}}=\Phi''(\rho_{\mu})\rho_{\text{nr}}-{4\pi}(-\Delta)^{-1}\rho_{\text{nr}}=0$, we have $\left(-\Delta-{4\pi\over \Phi''(\rho_{\mu})}\right)\psi_{\text{nr}}=0$. By  Lemma 3.5 (i)-(ii) and Lemma 3.4 in \cite{LZ2022},
\begin{align*}
&\ker\left(\left(-\Delta-{4\pi\over \Phi''(\rho_{\mu})}\right)|_{\{\psi\in\dot{H}^1(\mathbb{R}^3)|\int_{\mathbb{S}^2}\psi(r\theta)dS_\theta=0\}}\right)=\{\partial_{x^i}V_{\mu}, i=1,2,3\},\\
&\ker\left(L_{\mu}|_{Y_{\mu,\text{nr}}}\right)=\{\Delta\partial_{x^i}V_{\mu}, i=1,2,3\},
\end{align*}
where  $Y_{\mu,\text{nr}}:=\{\rho\in Y_{\mu}|\int_{\mathbb{S}^2}\rho(r\theta)dS_\theta=0\}$. Since
$\int_{B_{\mu}}\Delta\partial_{x^i}V_{\mu}dx=4\pi \partial_{x_i}\int_{B_{\mu}}\rho_{\mu}dx=0$, we have $\Delta\partial_{x^i}V_{\mu}\in Z_{\mu,\text{nr}}$, and thus, $\ker\left(L_{\mu}|_{Z_{\mu,\text{nr}}}\right)=\{\Delta\partial_{x^i}V_{\mu}, i=1,2,3\}$. In particular, we have $\int_{B_{\mu}}\rho_{\text{nr}}dx=0$.

Since $L_{\mu} \rho_{\text{r}}=c$, by Theorem 1.2 (iii) in \cite{LZ2022} we have $\rho_{\text{r}}=C\partial_\mu\rho_\mu$ for some constant $C$.  Since $\int_{B_{\mu}}\rho_{\text{r}}dx=\int_{B_{\mu}}\rho dx-\int_{B_{\mu}}\rho_{\text{nr}}dx=0$ and $M'(\mu)\neq0$, we have $0=\int_{B_{\mu}}\rho_{\text{r}}dx=C\partial_\mu\int_{B_{\mu}}\rho_\mu dx=CM'(\mu_{\mu})$, which implies that $C=0$ and $\rho_{\text{r}}=0$.
\end{proof}

\begin{lemma}\label{B0positive}
If the number of unstable modes $n^u(\mu)=0$ and $M'(\mu)\neq0$, then
\begin{align}\label{positive bound}
 \langle B''(0)\phi,\phi\rangle\geq C_0\|\nabla\phi\|^2_{L^2(\mathbb{R}^3)},\quad\forall\; \phi\in (\ker \tilde{L}_\mu)^\perp,
  \end{align}
  for some $C_0>0$.
\end{lemma}
\begin{proof}
If $n^u(\mu)=0$ and $M'(\mu)\neq0$,
then by Theorem 2.1 in \cite{LZ2022} we have
$n^-(L_\mu|_{Z_\mu})=n^u(\mu)=0$.
Thanks to Lemma \ref{negative-zero-direction}, one has
$n^-(\tilde{L}_\mu)=n^-(L_\mu|_{Z_\mu})=0$.
Let $W_\mu=\{\phi|\int_{B_{\mu}}{1\over \Phi''(\rho_{\mu})}(\phi-P\phi)^2dx<\infty\}$. Then $\dot{H}^1(\mathbb{R}^3)$  is embedded in $W_\mu$.
Let $\{\phi_n\}_{n=1}^\infty$ be a bounded sequence in  $\dot{H}^1(\mathbb{R}^3)$. By the compactness of $\dot{H}^1(\mathbb{R}^3)\hookrightarrow L^2(B_\mu)$, there exists $\phi_0\in L^2(B_\mu)$ such that $\phi_n\to\phi_0$ in  $L^2(B_\mu)$.
Then $\phi_n, \phi_0\in W_\mu$ and
\begin{align*}
\int_{B_{\mu}}{1\over \Phi''(\rho_{\mu})}(\phi_n-P\phi_n-\phi_0+P\phi_0)^2dx\leq \|\phi_n-\phi_0\|_{L^2(B_\mu)}^2\to0.
\end{align*}
Thus, $\dot{H}^1(\mathbb{R}^3)$  is compactly  embedded in $W_\mu$. For $\varphi\in \ker \tilde{L}_\mu$ and $\phi\in \dot{H}^1(\mathbb{R}^3)$, we have $(\phi,\varphi)_{W_\mu}=(\phi,\varphi)_{\dot{H}^1(\mathbb{R}^3)}$ due to $P\varphi=0$.
Then
\begin{align*}
\int_{B_{\mu}}{1\over \Phi''(\rho_{\mu})}(P\phi-\phi)^2 dx\leq \tilde C_0\|\nabla\phi\|_{L^2(\mathbb{R}^3)}^2,\quad \phi\in (\ker \tilde{L}_\mu)^\perp
\end{align*}
for some $0<\tilde C_0<1$. Then proves \eqref{positive bound}.
\end{proof}

\section{Conditional nonlinear orbital stability for general perturbations}
To remove the kernel of $\tilde L_\mu$, we choose the appropriate translations of the  perturbation of the
gravitational potential such that it is perpendicular to the  kernel functions.
\begin{lemma}\label{perplemma}
There exists $\delta_0>0$ such that for any $x_0\in\mathbb{R}^3$ and $(\rho,v)\in X$ with
$$d_4((\rho,v),(\rho_\mu(x+x_0),0))=\frac{1}{8\pi}\int_{\mathbb{R}^3}|\nabla V_{\text{in}}-\nabla V_{\mu}(x+x_0)|^2dx<\delta_0,$$ there exists $y_0\in \mathbb{R}^3$, depending continuously on $x_0\in\mathbb{R}^3$ and $\rho$, such that $ V_{\textup{in}}(x-y_0)-V_{\mu}\perp \partial_{x^i}V_{\mu}$ in $\dot{H}^1(\mathbb{R}^3)$ for $i=1,2,3,$ and $|x_0-y_0|\leq C\sqrt{\delta_0}$, where $V_{\textup{in}}=4\pi\Delta^{-1}\rho$.

\end{lemma}
\begin{proof}
For $x_0=0$, we define the map $f=(f_1,f_2,f_3): X\times \mathbb{R}^3\to \mathbb{R}^3$ by
\begin{align*}
f_i((\rho,v),y)=&\int_{\mathbb{R}^3}\nabla(V_{\text{in}}(x)-V_{\mu}(x+y))\cdot\nabla\partial_{x^i}V_{\mu}(x+y)dx\\
=&\int_{\mathbb{R}^3}\nabla(V_{\text{in}}(x-y)-V_{\mu}(x))\cdot\nabla\partial_{x^i}V_{\mu}(x)dx,\quad i=1,2,3.
\end{align*}
Then $f((\rho_{\mu},0),0)=0$.
By a direct computation, we have
\begin{align*}
\int_{\mathbb{R}^3}\nabla \partial_{x^i}V_\mu(|x|)\cdot \nabla \partial_{x^j}V_\mu(|x|) dx=0,
\end{align*}
for any $i\neq j $. Hence, we obtain that
\begin{align*}
&\frac{\partial (f_1,f_2,f_3)}{\partial(y^1,y^2,y^3)}\bigg|_{(\rho,v)=(\rho_{\mu},0),y=0}
=\\
&\left|
\begin{array}{ccc}
-\int_{\mathbb{R}^3}|\nabla \partial_{x^1}V_\mu(|x|)|^2 dx & 0 & 0 \\
0 & -\int_{\mathbb{R}^3}|\nabla \partial_{x^2}V_\mu(|x|)|^2 dx & 0 \\
0 & 0 & -\int_{\mathbb{R}^3}|\nabla \partial_{x^3}V_\mu(|x|)|^2 dx \\
\end{array}
\right|\neq 0.
\end{align*}
By the Implicit Function Theorem,  there exists $\delta>0$ such that for $(\rho,v)\in X$ with
$d_4((\rho,v),(\rho_\mu,0))<\delta$, one can find $y_0\in\mathbb{R}^3$ such that $f((\rho,v),y_0)=0$ and $y_0$  depends continuously on $(\rho,v)$.
The proofs of $|x_0-y_0|\leq C\sqrt{\delta}$ is due to the contraction of the map $T$ defined by
\begin{align*}
(T\eta)(\rho,v)=\eta(\rho,v)-\left(\frac{\partial (f_1,f_2,f_3)}{\partial(y^1,y^2,y^3)}\bigg|_{(\rho,v)=(\rho_{\mu},0),y=0}\right)^{-1}f((\rho,v),\eta(\rho,v)),
\end{align*}
where $\eta\in C(\bar{B}_{d_4}((\rho_{\mu},0),\delta),\mathbb{R}^3)$.

Finally, we consider the case $x_0\neq 0$.  For $(\rho,v)\in X$ with $d_4((\rho,v),(\rho_\mu(x+x_0),0))<\delta$, we have $d_4(( \rho(x-x_0),v),(\rho_\mu,0))<\delta$. Thus, there exists $\tilde y_0\in\mathbb{R}^3$ such that $f((\rho,v),x_0+\tilde y_0)=f((\rho(x-x_0),v),\tilde y_0)=0$.
\end{proof}

Now, we are in a position to prove  the conditional nonlinear orbital stability of the non-rotating stars for general perturbations.

\begin{proof}[Proof of Theorem \ref{main}]
For any $\ep>0$, let $\delta<\min\{{\ep\over 2C_1}, {\delta_0\over 2},1\}$, where $C_1>1$ and $\delta_0>0$ are given in \eqref{p-estimate} and   Lemma \ref{perplemma}, respectively.
For the initial data $(\rho(0),v(0))$ satisfying \eqref{initial data-1}, we choose $x_0(0)\in\mathbb{R}^3$ such that
\begin{align}\label{inf-tran}
d(({\rho}(0),{v}(0)),(\rho_\mu(x+x_0(0)),0))+|{M}-M_{\mu}|^q<\delta.
\end{align}

For $t\geq0$, we prove that if there exists $x_0(t)\in\mathbb{R}^3$ satisfying $d(({\rho}(t),{v}(t)),(\rho_\mu(x+x_0(t)),0))<\delta_0$, then
there exists $y_0(t)\in\mathbb{R}^3$ such that
 \begin{align}\label{pr-estimate}
 d(({\rho}(t),{v}(t)),(\rho_\mu(x+y_0(t)),0))<\epsilon.
 \end{align}
  Indeed, by Lemma  \ref{perplemma}, there exists $y_0(t)\in\mathbb{R}^3$ such that $ V_{\textup{in}}(x-y_0(t))-V_{\mu}\perp \ker\left(\tilde L_\mu\right)$ in $\dot{H}^1(\mathbb{R}^3)$, $y_0(t)$ depends continuously on $t$ and $|x_0(t)-y_0(t)|<C\sqrt{\delta_0}$. For $t=0$, $y_0(0)$ can be chosen such that  $|x_0(0)-y_0(0)|<C\sqrt{\delta}$ by \eqref{inf-tran}.
Then $$d((\rho_\mu(x+x_0(0)),0),(\rho_\mu(x+y_0(0)),0))<{\ep\over 2C_1}$$ by choosing $\delta>0$ small enough and
\begin{align}\label{ker-tran}
&d(({\rho}(0),{v}(0)),(\rho_\mu(x+y_0(0)),0))+|{M}-M_{\mu}|^q
\leq d(({\rho}(0),{v}(0)),(\rho_\mu(x+x_0(0)),0))\\\nonumber
&+d((\rho_\mu(x+x_0(0)),0),(\rho_\mu(x+y_0(0)),0))+|{M}-M_{\mu}|^q<{\ep\over C_1}.
\end{align}

By Lemma \ref{B0positive},
$
 \langle B''(0)\phi,\phi\rangle\geq C_0\|\nabla\phi\|^2_{L^2(\mathbb{R}^3)}$ for  $\phi\in (\ker \tilde{L}_\mu)^\perp$, where $C_0>0$.
Then by  \eqref{H-rho-vHrhomu0}, \eqref{ker-tran}, \eqref{d2-d4-estimate},  \eqref{dual functional expression} and Lemmas \ref{high order term EC functional}, \ref{c2dual functional},
\ref{B0positive}, we have
\begin{align}\label{d0dt-estimate}
&Cd(({\rho}(0),{v}(0)),(\rho_\mu(x+y_0(0)),0))\\\nonumber
\geq&H(\rho(0),v(0))-H(\rho_{\mu}(x+y_0(0)),0)\geq H(\rho_{tran}(t),v(t))-H(\rho_{\mu},0)\\\nonumber
\geq& \sum_{i=1,3,5}d_{i,tran}(t)+\tau(d_{2,tran}(t)-d_{4,tran}(t))\\\nonumber
&+(1-\tau)(d_{2,tran}(t)-d_{4,tran}(t))-o(d_{tran}(t))\\\nonumber
\geq &\sum_{i=1,3,5}d_{i,tran}(t)+\tau(d_{2,tran}(t)-d_{4,tran}(t))+(1-\tau)\bigg( B(\tilde{V}_{\text{tran,in}}(t))\\\nonumber
&-C d_{tran}(t)^{\frac{1}{2}}(|M-M_{\mu}|+d_{tran}(t)^{1\over\gamma_{0}}
+d_{tran}(t)^{1\over\gamma_1}+d_{tran}(t))\bigg)-o(d_{tran}(t))\\\nonumber
\geq& \sum_{i=1,3,5}d_{i,tran}(t)+\tau(d_{2,tran}(t)-d_{4,tran}(t))\\\nonumber
&+(1-\tau)B(\tilde{V}_{\text{tran,in}}(t))
 -C |M-M_{\mu}|^q-o(d_{tran}(t))\\\nonumber
 =& \sum_{i=1,3,5}d_{i,tran}(t)+\tau(d_{2,tran}(t)-d_{4,tran}(t))\\\nonumber
&+(1-\tau)(\langle B''(0)\tilde{V}_{\text{tran,in}}(t),\tilde{V}_{\text{tran,in}}(t)\rangle+o(d_{4,tran}(t)))
 -C |M-M_{\mu}|^q-o(d_{tran}(t))\\\nonumber
 \geq& \sum_{i=1,3,5}d_{i,tran}(t)+\tau d_{2,tran}(t)
+( (1-\tau)C_0-\tau)d_{4,tran}(t)
 -C |M-M_{\mu}|^q-o(d_{tran}(t))\\\nonumber
 \geq&\tau d_{tran}(t)
 -C |M-M_{\mu}|^q-o(d_{tran}(t))
\end{align}
for $\tau>0$ small enough,
where $$d_{tran}(t)=d((\rho_{tran}(t),v(t)),(\rho_\mu,0)),\quad d_{i,tran}(t)=d_i((\rho_{tran}(t),v(t)),(\rho_\mu,0))$$
 for $i=1,\cdots,5$, 
 $$\rho_{tran}(t)=\rho(t,x-y_0(t)), \quad \tilde{V}_{\text{tran,in}}(t)=V_{\text{tran,in}}(t)-V_{\mu},$$ and $1<q<2$. Then by \eqref{ker-tran}, we have
\begin{align}\label{p-estimate}
 d(({\rho}(t),{v}(t)),(\rho_\mu(x+y_0(t)),0))=
d_{tran}(t)\leq C_1 (d_{tran}(0)+|M-M_{\mu}|^q)<\ep,
\end{align}
where we used the continuity of $d_{tran}(t)$ on $t$. This proves \eqref{pr-estimate}.

Let $\ep<\delta_0$. If there exists $t_0>0$ such that
$\inf_{y\in \mathbb{R}^3}d((\rho(t_0),v(t_0)),(\rho_\mu(x+y),0))=\ep$, then there exists $x_0(t_0)\in\mathbb{R}^3$ such that $\ep\leq d((\rho(t_0),v(t_0)),(\rho_\mu(x+x_0(t_0)),0))<\delta_0$. By \eqref{pr-estimate}, there exists $y_0(t_0)\in\mathbb{R}^3$ such that
$ d(({\rho}(t_0),{v}(t_0)),(\rho_\mu(x+y_0(t_0)),0))<\epsilon,
$ which is a contradiction.
\end{proof}

%If a weak solution satisfies the following entropy conditions,
%\begin{definition}\label{weaksol}
%A weak solution (defined above) on $[0, T ]\times \mathbb{R}^3$ is called an entropy weak
%solution of \eqref{EP} if it satisfies the following ``entropy inequality'':
%$$\int_0^T\int \left(\eta\beta_t +q \cdot\nabla \beta-\rho\sum_{i=1}^3\eta_{m_j}V_{x_j}\beta\right)dxdt+\int \beta(0,x)\eta(0,x)dx\geq0,$$
%for any nonnegative Lipschitz continuous test function $\beta$ with compact support in $[0, T )\times \mathbb{R}^3$. Here the ``entropy'' function $\eta$ and ``entropy flux'' functions $q_j$ and $q$, are defined
%by
%\begin{align*}
%\begin{cases}
%\eta=\frac{|m|^2}{2\rho}+\rho\int_0^\rho\frac{p(s)}{s^2}ds,\\
%q_j=\frac{|m|^2m_j}{2\rho^2}+m_j\int_0^\rho\frac{p'(s)}{s}ds,\\
%q=(q_1,q_2,q_3).
%\end{cases}
%\end{align*}
%\end{definition}

\section{Unconditional nonlinear stability of non-rotating star for spherically symmetric perturbations}
In this section, we prove the unconditional nonlinear stability of the non-rotating stars for a class of general pressure laws satisfying (C1)-(C3) under spherically symmetric perturbations.

When the perturbation is restricted to be spherically symmetric, the translation kernel of $\tilde L_\mu$ in Lemma \ref{kerLmu} can be removed.

\begin{lemma}\label{trivial kernel}
Suppose that $P$ satisfies (C1)-(C3). For $\mu\in (0,\mu_{max})$ and
any non-rotating stars $(\rho_\mu,0)$ satisfying  $n^u(\mu)=0$ and $M'(\mu)\neq0$, then we have
\begin{align*}
 \ker(\tilde L_\mu^r)=\{0\},
\end{align*}
and
\begin{align}\label{positive bound-radial}
 \langle B''(0)\phi,\phi\rangle\geq C\|\nabla\phi\|^2_{L^2(\mathbb{R}^3)},\quad \phi\in \dot{H}_r^1,
 \end{align}
where $\tilde L_\mu^r=\tilde L_\mu|_{\dot{H}_r^1}$ and $\dot{H}_r^1=\{f(x)\in\dot{H}^1(\mathbb{R}^3)| f(x)=f(|x|)\}$.
\end{lemma}
\begin{proof}
%For polytropic  stars, by \cite{CS1939} we know that $M(\mu)= C \mu^\frac{3\gamma-4}{2}$ for some constant $C>0$.
%Thus, $M'(\mu)=C\frac{3\gamma-4}{2}\mu^\frac{3\gamma-6}{2}>0$ for $\mu\in (0,+\infty)$.
By Lemma \ref{kerLmu},  $\ker\left(\tilde L_\mu\right)=\{\partial_{x_i}V_{\mu}, i=1,2,3\}.$
Since $\partial_{x_i}V_{\mu}\perp \dot{H}_r^1$ in $\dot{H}^1(\mathbb{R}^3)$, we have $\ker(\tilde L_\mu^r)=\{0\}$ and
 \eqref{positive bound-radial} holds  by Lemma \ref{B0positive}.
\end{proof}

Now, we give the proof of Theorem \ref{unconditionalTHM}.

\begin{proof} [Proof of Theorem \ref{unconditionalTHM}] We divide the proof into $5$ steps.

Step $1$. Construction of the approximate initial data.

For the  spherically symmetric initial data $(\rho(0),v(0), V(0))$,  one can construct a sequence of approximate initial data $(\rho^{\varepsilon}(0), v^{\varepsilon}(0), V^\varepsilon(0)), 0<\varepsilon\ll1,$  by  a similar argument  to (A.8), (A.27) and (A.12)   of \cite{CHWY2021}.
Moreover, $(\rho^{\varepsilon}(0), v^{\varepsilon}(0), V^\varepsilon(0))$ satisfies
\begin{align}\nonumber
&\rho^{\varepsilon}(0)\to\rho(0) \quad \text{in} \quad L^{1}\cap L^{\gamma_1}(\mathbb{R}^3),\\\nonumber %\cite{CHWW2023}(2.9)下面一段
&V^\varepsilon(0)\to V(0)\quad \text{in}\quad \dot{H}^1(\mathbb{R}^3),\\%\cite{CHWW2023}(2.9)下面一段
&E_0^\varepsilon\to E_0\quad \text{as}\quad\varepsilon\to0,\label{first two part of energy}%\cite{CHWW2023}(2.9)下面一段
\end{align}
 and $\int_{\mathbb{R}^3}\rho^\varepsilon(0)dx=M$, where $E_0^\varepsilon:=\int_{\mathbb{R}^3}\left(\frac{1}{2}\rho^\varepsilon(0) (v^\varepsilon(0))^2+\Phi(\rho^\varepsilon(0))\right) dx$.
For  the  approximate initial data $(\rho^\varepsilon(0),v^\varepsilon(0), V^\varepsilon(0))$, we  can construct a sequence of smooth initial data $(\rho^{\varepsilon,b}(0),v^{\varepsilon,b}(0),V^{\varepsilon,b}(0))$
by  a similar argument  to (A.41) and (A.25) of \cite{CHWY2021}. Moreover, $(\rho^{\varepsilon,b}(0),v^{\varepsilon,b}(0),V^{\varepsilon,b}(0))$ satisfies
\begin{align}\label{initial-rho-approximate-cnsp-b}
&\rho^{\varepsilon,b}(0)\to\rho^\varepsilon(0) \quad \text{in} \quad L^{1}\cap L^{\gamma_1}(B(0,b)),\\%\cite{CHWW2023}(5.10)
%\label{initial-potential-approximate-cnsp-b}&V^{\varepsilon,b}(0)\to V^{\varepsilon}(0)\quad \text{in}\quad \dot{H}^1(\mathbb{R}^3),\\
\label{initial-approximate-cnsp}
&\quad E_0^{\varepsilon,b}\to E_0^{\varepsilon} \quad\text{as}\quad b\to\infty,%\cite{CHWW2023}（5.10）
\end{align}
 and $\int_{\mathbb{R}^3}\rho^{\varepsilon,b}(0)dx=M,$
where $E_0^{\varepsilon,b}:=\int_{\mathbb{R}^3}\left(\frac{1}{2}\rho^{\varepsilon,b}(0) (v^{\varepsilon,b}(0))^2+\Phi(\rho^{\varepsilon,b}(0))\right) dx$.

Step $2$. Estimates from nonlinear stability of the approximate solutions $(\rho^{\varepsilon,b}, v^{\varepsilon,b})$  to
 the approximate free boundary value problems for
 the  compressible Navier-Stokes-Poisson (CNSP) systems.

The free boundary value problem for the CNSP system is
\begin{align}\label{NSP}
\begin{cases}
\partial_t\rho + \text{div}(\rho v) = 0,\\
\partial_t(\rho v)+\text{div}(\rho v\otimes v)+\nabla p=-\rho\nabla V+\varepsilon div(\rho D(v)),\\
\Delta V=4\pi\rho,\;\;\lim_{|x|\to\infty}V(t,x)=0,
\end{cases}
\end{align}
where $D(v)=\frac{1}{2}(\nabla v+(\nabla v)^\perp)$ is the stress tensor,  $\varepsilon>0$ is the inverse of the Reynolds number, $(t,r)\in \Omega_T=\{(t,r): a\leq r\leq b(t),0\leq t\leq T\}$  for any $T>0$, $\{r=b(t):0<t\leq T\}$ is a free boundary determined by $b'(t)=u(t,b(t))$ for $t>0$ and $b(0)=b$, $u{x\over r}=v$, $a=b^{-1}$. The stress-free boundary condition $(p(\rho)-\varepsilon\rho(u_r+{2\over r}u))(t,b(t))=0$  is posed on the free boundary $r=b(t)$ for $t>0$, and $u=0$ is posed on the fixed boundary $r=a$.

Let $T>0$ and  $(\rho^{\varepsilon,b}(t), v^{\varepsilon,b}(t),V^{\varepsilon,b}(t))$ be a regular solution to the free boundary value problem \eqref{NSP} for $t\in[0,T]$. Here,
$(\rho^{\varepsilon,b}(t), v^{\varepsilon,b}(t))$ is understood as a solution on $\mathbb{R}^3$ by zero extension outside  $\{x|a\leq |x|\leq b(t)\}$ for $t\in[0,T]$, and
the potential $V^{\varepsilon,b}(t)$ is given by
$  V^{\varepsilon,b}(t)=4\pi\Delta^{-1}\rho^{\varepsilon,b}(t)$ for $t\in[0,T]$.
By
 \begin{align*}
&E(\rho^{\varepsilon,b}(t),v^{\varepsilon,b}(t))\leq E(\rho^{\varepsilon,b}(0),v^{\varepsilon,b}(0)),\quad t\in[0,T],
\end{align*} and the conversation of mass, we have
\begin{align}\label{ec-decrease} H(\rho^{\varepsilon,b}(t),v^{\varepsilon,b}(t))\leq H(\rho^{\varepsilon,b}(0),v^{\varepsilon,b}(0)),
\end{align}
where $H$ is defined in \eqref{Energy-Casimir functional}.
  Note that $\int_{\mathbb{R}^3}\rho^{\varepsilon,b}(0) dx=\int_{\mathbb{R}^3}\rho(0) dx={M}$. Denote $d_{\varepsilon,b}(t)=d((\rho^{\varepsilon,b}(t),v^{\varepsilon,b}(t)),(\rho_\mu,0))$, $t\in[0,T]$ for simplicity.
There exists  $\delta_1\in(0, 1)$ such that if
\begin{align*}
d_{\varepsilon,b}(0)+|{M}-M_{\mu}|^q<\delta_1,
\end{align*}
then by \eqref{ec-decrease} and Lemma \ref{trivial kernel},  a similar but simpler argument than \eqref{d0dt-estimate} implies
\begin{align}\label{symmetric-estimate}
&Cd_{\varepsilon,b}(0)
\geq H(\rho^{\varepsilon,b}(0),v^{\varepsilon,b}(0))-H(\rho_{\mu},0)\geq H(\rho^{\varepsilon,b}(t),v^{\varepsilon,b}(t))-H(\rho_{\mu},0)\\\nonumber
 \geq&\tau d_{\varepsilon,b}(t)
 -C |M-M_{\mu}|^q-o(d_{\varepsilon,b}(t))
\end{align}
for $\tau>0$ small enough.
Thus,
\begin{align}\label{symmetric-p-estimate}
 d(({\rho}^{\varepsilon,b}(t),{v}^{\varepsilon,b}(t)),(\rho_\mu,0))=
d_{\varepsilon,b}(t)\leq C_1 (d_{\varepsilon,b}(0)+|M-M_{\mu}|^q),\quad t\in[0,T],
\end{align}
where we used the continuity of $d_{\varepsilon,b}(t)$ on $t\in[0,T]$. The proof of
\eqref{symmetric-estimate} is simpler than \eqref{d0dt-estimate} since
the kernel of
$
\tilde L_\mu^r
$ is trivial
and
 $ \langle B''(0)\cdot,\cdot\rangle$ is positive on $\dot{H}_r^1$ by Lemma \ref{trivial kernel}. Thus,
 we do not need to remove the  kernel by translations.

Step $3$. Uniform estimates and global existence of finite-energy weak solutions to the CEP systems.

The local existence of the unique  classical solution to the approximate free boundary value problem for
 the  CNSP system \eqref{NSP}
could  be shown by a standard argument as in \cite{JXZ2005}. By \eqref{symmetric-p-estimate}, we have the uniform energy estimate
\begin{align}\label{symmetric-p-estimate2}
\int_{\mathbb{R}^3}\Phi(\rho^{\varepsilon,b}(t))dx\leq C_1 (d_{\varepsilon,b}(0)+|M-M_{\mu}|^q)\leq C_1\delta_1
\end{align}
for the regular solutions $(\rho^{\varepsilon,b}(t), v^{\varepsilon,b}(t))$, $t\in[0,T]$, to the free boundary value problems \eqref{NSP}. Then we get  similar
 a priori estimates as established in Lemmas 3.3-3.12 of \cite{DL2015}. A continuity argument implies that the smooth solution $(\rho^{\varepsilon,b}, v^{\varepsilon,b})$  exists globally in time.
Thanks to \eqref{symmetric-p-estimate2}, we can establish the same uniform estimates, independent of $b$,  as in Section 5 of \cite{CHWW2023}. Then using the same compactness arguments in \cite{CHWW2023} (see also \cite{CHWY2021,CW2022}), we can take the limit $b\to\infty$ to obtain the global weak solutions $(\rho^{\varepsilon}, v^{\varepsilon})$ of the
 CNSP system \eqref{NSP} with the initial data $(\rho^\varepsilon(0),v^\varepsilon(0))$. Finally, by the $L^p$ compensated  compactness framework established in \cite{CHWW2023}, we can take the inviscid limit $\varepsilon\to0$ and obtain a global  finite-energy weak solution $(\rho, v)$ to the CEP system \eqref{EP}-\eqref{inaprx} with the initial data $(\rho(0),v(0), V(0))$.
More precisely,
there exist both a subsequence of (which is still denoted by) $(\rho^\varepsilon, v^\varepsilon, V^\varepsilon)$ and a  spherically symmetric vector function $(\rho,v,V)$ such that
\begin{align}\label{rho-convergence22}
(\rho^\varepsilon, m^\varepsilon)(t,r)\to(\rho,m)(t,r)\quad\text{in}\quad L_{loc}^{q_1}(\mathbb{R}_+^2)\times L_{loc}^{q_2}(\mathbb{R}_+^2),%\cite{CHWW2023} 的(2.31)
\end{align}
with $q_1\in[1,\gamma_1 +1)$, $q_2\in[1,{3(\gamma_1+1)\over \gamma_1 +3})$, and
\begin{align*}
&V^\varepsilon\rightharpoonup V\quad\text{in}\quad L^2 ([0,T], H_{loc}^1(\mathbb{R}^3)),\\%\cite{CHWW2023}（9.8）的上面
&\int_0^\infty|V_r^\varepsilon(t,r)-V_r(t,r)|^2r^2dr\to 0 \quad \text{as} \quad \varepsilon\to0,
\end{align*}
 where $(\rho v)(t,r)=m(t,r){x\over r}$. Moreover,
$(\rho,v,V)$ is  a globally spherically symmetric  finite-energy weak solution  of the CEP system \eqref{EP}-\eqref{inaprx} with the initial data  $(\rho(0),v(0), V(0))$  in the sense of Definition \ref{weaksol}.

Step $4$. Prove the nonlinear stability for the global weak solutions $(\rho^{\varepsilon}, v^{\varepsilon})$  of the  CNSP systems \eqref{NSP}.
Let   $d_{\varepsilon}(t)=d((\rho^{\varepsilon}(t),v^{\varepsilon}(t)),(\rho_\mu,0))$, $t\geq0$. The mass satisfies $\int_{\mathbb{R}^3}\rho^{\varepsilon}(0) dx={M}$.
Let
$d_{\varepsilon}(0)+|{M}-M_{\mu}|^q<{1\over3}\delta_1$. Now, fix $\varepsilon$. For any $\kappa>0$, there exists $b_1>0$ such that $\int_{b}^\infty |\rho^\varepsilon(0)|^s r^2dr<{\kappa\over2}$ and $\int_0^b |\rho^{\varepsilon,b}(0)-\rho^\varepsilon(0)|^sr^2dr<{\kappa\over2}$ for $b>b_1$, where $s=1,\gamma_1$ and
we used $\eqref{initial-rho-approximate-cnsp-b}$. Thus,
$
\int_0^\infty |\rho^{\varepsilon,b}(0)-\rho^\varepsilon(0)|^sr^2dr<\kappa$ for $b>b_1$. This implies that
\begin{align}\label{rho-varepsilon-b-tend infty}
\rho^{\varepsilon,b}(0)\to\rho^\varepsilon(0) \quad \text{in} \quad L^{1}\cap L^{\gamma_1}(\mathbb{R}^3)
\end{align}
as $b\to \infty$.
 Note that
\begin{align}\label{initial-potential-in-approximate-cnsp-b}
&{1\over4\pi}\int_{\mathbb{R}^3}|\nabla V_{\text{in}}^{\varepsilon,b}(0)-\nabla V_{\text{in}}^{\varepsilon}(0)|^2dx
=-\int_{\mathbb{R}^3}( V_{\text{in}}^{\varepsilon,b}(0)- V_{\text{in}}^{\varepsilon}(0))(\rho_{\text{in}}^{\varepsilon,b}(0)-\rho_{\text{in}}^{\varepsilon}(0))dx\\\nonumber
\leq& \|V_{\text{in}}^{\varepsilon,b}(0)- V_{\text{in}}^{\varepsilon}(0)\|_{L^6(\mathbb{R}^3)}\|\rho_{\text{in}}^{\varepsilon,b}(0)-\rho_{\text{in}}^{\varepsilon}(0)\|_{L^{6/5}(\mathbb{R}^3)}\\\nonumber
\leq&C \|\rho^{\varepsilon,b}(0)-\rho^{\varepsilon}(0)\|_{L^{6/5}(B_{\mu})}^2\to0
 \end{align}
 as $b\to\infty$.
 Since $|\Phi'(\rho_\mu)|\leq C_{\mu}$ on $B_{\mu}$ and $|V_\mu-V_\mu(R_{\mu})|\leq 2|V_\mu(R_{\mu})|\leq  C_{\mu}$ on $B_{\mu}^c$, by  \eqref{initial-approximate-cnsp}, \eqref{rho-varepsilon-b-tend infty} and \eqref{initial-potential-in-approximate-cnsp-b} we have
 \begin{align*}
 &{1\over2}\int_{\mathbb{R}^3}(\rho^{\varepsilon,b}(0)|v^{\varepsilon,b}(0)|^2-\rho^{\varepsilon}(0)|v^{\varepsilon}(0)|^2)dx+
 \int_{\mathbb{R}^3}(\Phi(\rho^{\varepsilon,b}(0))-\Phi(\rho^{\varepsilon}(0)))dx\\
 &-\int_{B_\mu}\Phi'(\rho_\mu)(\rho^{\varepsilon,b}(0)-\rho^{\varepsilon}(0))dx+{1\over4\pi}\int_{\mathbb{R}^3}|\nabla V_{\text{in}}^{\varepsilon,b}(0)-\nabla V_{\text{in}}^{\varepsilon}(0)|^2dx\\
 &+\int_{B_\mu^c}(V_\mu-V_\mu(R_\mu))(\rho^{\varepsilon,b}(0)-\rho^{\varepsilon}(0))dx\to0
 \end{align*}
 as $b\to\infty$. Then
 \begin{align}\label{d-varepsilon-b-0-to-d-varepsilon-0}
 d_{\varepsilon,b}(0)\to  d_{\varepsilon}(0)
 \end{align}
 as $b\to\infty$, and
  for $b>0$ large enough,
 \begin{align*}
 &d_{\varepsilon,b}(0)+|{M}-M_{\mu}|^q\\
 \leq&
 {1\over2}\int_{\mathbb{R}^3}(\rho^{\varepsilon,b}(0)|v^{\varepsilon,b}(0)|^2-\rho^{\varepsilon}(0)|v^{\varepsilon}(0)|^2)dx+
 \int_{\mathbb{R}^3}(\Phi(\rho^{\varepsilon,b}(0))-\Phi(\rho^{\varepsilon}(0)))dx\\
 &-\int_{B_\mu}\Phi'(\rho_\mu)(\rho^{\varepsilon,b}(0)-\rho^{\varepsilon}(0))dx+{1\over4\pi}\int_{\mathbb{R}^3}|\nabla V_{\text{in}}^{\varepsilon,b}(0)-\nabla V_{\text{in}}^{\varepsilon}(0)|^2dx\\
 &+\int_{B_\mu^c}(V_\mu-V_\mu(R_\mu))(\rho^{\varepsilon,b}(0)-\rho^{\varepsilon}(0))dx+\sum_{i=1,2,3,5}d_{i,\varepsilon}(0)+2d_{4,\varepsilon}(0)+|{M}-M_{\mu}|^q<\delta_1,
 \end{align*}
where $d_{i,\varepsilon}(0)=d_{i,\varepsilon}((\rho^{\varepsilon}(0),v^{\varepsilon}(0)),(\rho_\mu,0))$ for $i=1,\cdots, 5$. By \eqref{symmetric-p-estimate}, we have
\begin{align}\label{symmetric-p-estimate-app}
d_{\varepsilon,b}(t)\leq C_1 (d_{\varepsilon,b}(0)+|M-M_{\mu}|^q), \quad t>0,
\end{align}
for $b>0$ large enough. By Lemma 6.1 in \cite{CHWW2023}, we have  $\rho^{\varepsilon,b}(t)\to\rho^\varepsilon(t)$ in $L^{\hat{q}}(D_R)$ for any $\hat{q}\geq1$, $t>0$ and $R>0$, where $D_R=\{x|R^{-1}\leq |x|\leq R\}$. By Lemma 5.1 in \cite{CHWW2023},
$\sqrt{\rho^{\varepsilon,b}(t)}v^{\varepsilon,b}(t)\rightharpoonup\sqrt{\rho^{\varepsilon}(t)}v^{\varepsilon}(t)$ in $L^2(D_R)$ for $t>0$.
By Lemma 5.1 in \cite{CHWW2023} again, a similar argument  to \eqref{initial-potential-in-approximate-cnsp-b} gives
$
{1\over4\pi}\int_{\mathbb{R}^3}|\nabla V_{\text{in}}^{\varepsilon,b}(t)|^2dx\leq C \|\rho^{\varepsilon,b}(t)\|_{L^{\gamma_1}(B_{\mu})}^2<C$,
which implies
$ V_{\text{in}}^{\varepsilon,b}(t)\rightharpoonup V_{\text{in}}^{\varepsilon}(t)$ in $\dot{H}^1(\mathbb{R}^3)$ for $t>0$. Since $\Phi$ is convex,
it follows from \cite{Morrey1966} that
\begin{align*}
&\frac{1}{2}\int_{D_R}\rho^{\varepsilon}(t)|v^{\varepsilon}(t)|^2dx +\int_{D_R}\Phi(\rho^{\varepsilon}(t))dx+
\int_{B_\mu}(-\Phi(\rho_\mu)+\Phi'(\rho_\mu)\rho_\mu) dx\\
&+\frac{1}{8\pi}\int_{\mathbb{R}^3}|\nabla V_{\text{in}}^{\varepsilon}(t)-\nabla V_\mu|^2dx+\int_{D_R} (V_\mu-V_\mu(R_{\mu}))\rho^{\varepsilon}(t) dx\\
\leq&\liminf_{b\to\infty}\bigg(\frac{1}{2}\int_{D_R}\rho^{\varepsilon,b}(t)|v^{\varepsilon,b}(t)|^2dx +\int_{D_R}\Phi(\rho^{\varepsilon,b}(t))dx+\frac{1}{8\pi}\int_{\mathbb{R}^3}|\nabla V_{\text{in}}^{\varepsilon,b}(t)-\nabla V_\mu|^2dx\\
&+\int_{D_R} (V_\mu-V_\mu(R_{\mu}))\rho^{\varepsilon,b}(t) dx\bigg)+
\int_{B_\mu}(-\Phi(\rho_\mu)+\Phi'(\rho_\mu)\rho_\mu) dx\\
\leq&\liminf_{b\to\infty}d_{\varepsilon,b}(t)\leq C_1 (d_{\varepsilon}(0)+|M-M_{\mu}|^q)
\end{align*}
for  $t>0$ and $R>0$, where we used \eqref{d-varepsilon-b-0-to-d-varepsilon-0}-\eqref{symmetric-p-estimate-app} in the last inequality above. Letting $R\to\infty$, we have
\begin{align}\label{symmetric-p-varepsilon-estimate-app}
d_{\varepsilon}(t)\leq C_1 (d_{\varepsilon}(0)+|M-M_{\mu}|^q),\quad t>0.
\end{align}

Step $5$.  Prove the nonlinear stability for the global weak solutions $(\rho, v)$  of the  CEP systems \eqref{EP}-\eqref{inaprx}.
Recall that  $d(t)=d((\rho(t),v(t)),(\rho_\mu,0))$, $t\geq0$ and $\int_{\mathbb{R}^3}\rho(0) dx={M}$.
Let
$d(0)+|{M}-M_{\mu}|^q<{1\over9}\delta_1$. By \eqref{rho-convergence22},
$
\rho^{\varepsilon}(0)\to\rho(0)$ in $ L_{loc}^{1}\cap L_{loc}^{\gamma_1}(\mathbb{R}^3)
$
as $\varepsilon\to 0$.
Similar to \eqref{initial-potential-in-approximate-cnsp-b}, we have
$
 V_{\text{in}}^{\varepsilon}(0)\to V_{\text{in}}(0)
 $ in $\dot{H}^1(\mathbb{R}^3)$
 as $\varepsilon\to0$.
This, along with \eqref{first two part of energy}, implies
that $ d_{\varepsilon}(0)\to  d(0)$ as $\varepsilon\to0$ and
$
 d_{\varepsilon}(0)+|{M}-M_{\mu}|^q<{1\over3}\delta_1
 $
for $\varepsilon>0$ small enough. By \eqref{symmetric-p-varepsilon-estimate-app}, we have
$
d_{\varepsilon}(t)\leq C_1 (d_{\varepsilon}(0)+|M-M_{\mu}|^q)$
for $\varepsilon>0$ small enough and  $t\geq0$.
By Lemma 6.3 and (6.6) in \cite{CHWW2023}, we have  $\rho^{\varepsilon}(t)\rightharpoonup\rho(t)$ in $L^1(\mathbb{R}^3)$, $\sqrt{\rho^{\varepsilon}(t)}v^{\varepsilon}(t)\rightharpoonup\sqrt{\rho(t)}v(t)$ in $L^2(\mathbb{R}^3)$ and $\|\Phi(\rho^{\varepsilon})(t)\|_{L^{1}(B_{\mu})}<C$  for $t>0$. This implies that  $ V_{\text{in}}^{\varepsilon}(t)\rightharpoonup V_{\text{in}}(t)$ in $\dot{H}^1(\mathbb{R}^3)$  for $t>0$.
It follows from \cite{Morrey1966} that
\begin{align*}
&\frac{1}{2}\int_{\mathbb{R}^3}\rho(t)|v(t)|^2dx +\int_{B_R}\Phi(\rho(t))dx+
\int_{B_\mu}(-\Phi(\rho_\mu)+\Phi'(\rho_\mu)\rho_\mu) dx\\
&+\frac{1}{8\pi}\int_{\mathbb{R}^3}|\nabla V_{\text{in}}(t)-\nabla V_\mu|^2dx+\int_{\mathbb{R}^3} (V_\mu-V_\mu(R_{\mu}))\rho(t) dx\\
\leq&\liminf_{\varepsilon\to0}\bigg(\frac{1}{2}\int_{\mathbb{R}^3}\rho^{\varepsilon}(t)|v^{\varepsilon}(t)|^2dx +\int_{B_R}\Phi(\rho^{\varepsilon}(t))dx+\frac{1}{8\pi}\int_{\mathbb{R}^3}|\nabla V_{\text{in}}^{\varepsilon}(t)-\nabla V_\mu|^2dx\\
&+\int_{\mathbb{R}^3} (V_\mu-V_\mu(R_{\mu}))\rho^{\varepsilon}(t) dx\bigg)+
\int_{B_\mu}(-\Phi(\rho_\mu)+\Phi'(\rho_\mu)\rho_\mu) dx\\
\leq&\liminf_{\varepsilon\to0}d_{\varepsilon}(t)\leq C_1 (d(0)+|M-M_{\mu}|^q)
\end{align*}
for  $t>0$, where $B_R=\{x| |x|\leq R\}$ and $R>0$. Sending $R\to\infty$, we have
$
d(t)\leq C_1 (d(0)+|M-M_{\mu}|^q)$
for $t>0$.
\end{proof}

\section*{Acknowledgement}
Z. Lin is
 partially supported by the NSF grants DMS-1715201 and DMS-2007457.  H. Zhu is  partially supported by National Key R $\&$ D Program of China under Grant 2021YFA1002400, NSF
of China under Grant 12101306 and NSF of Jiangsu Province, China under Grant BK20210169.

\end{CJK*}

\end{document}